\documentclass[12pt,reqno]{amsart}
\usepackage{amsfonts}
\usepackage{amssymb,amsmath}
\usepackage{amscd}
\usepackage{hyperref}
\usepackage{graphicx}
\usepackage{epsfig}
\usepackage{epstopdf}
\usepackage{caption}
\usepackage{subcaption}

\newtheorem{theorem}{Theorem}[section]
\newtheorem{lemma}{Lemma}[section]
\newtheorem{proposition}{Proposition}[section]
\newtheorem{corollary}{Corollary}[section]
\newtheorem{definition}{Definition}[section]
\newtheorem{example}{Example}[section]

\numberwithin{equation}{section}
\begin{document}
\title[A New Generalization of Rhoades' Condition]{A New Generalization of
Rhoades' Condition}
\author[N. TA\c{S}]{N\.{I}HAL TA\c{S}}
\address{Bal\i kesir University\\
Department of Mathematics\\
10145 Bal\i kesir, TURKEY}
\email{nihaltas@balikesir.edu.tr}
\author[N. \"{O}ZG\"{U}R ]{N\.{I}HAL \"{O}ZG\"{U}R}
\address{Bal\i kesir University\\
Department of Mathematics\\
10145 Bal\i kesir, TURKEY}
\email{nihal@balikesir.edu.tr}

\begin{abstract}
In this paper, our aim is to obtain a new generalization of the well-kown Rhoades' contractive
condition. To do this, we introduce the notion of an $S$-normed space. We extend the Rhoades' contractive condition
to $S$-normed spaces and define a new type of contractive conditions. We support our theoretical results with necessary illustrative examples.
\end{abstract}

\maketitle

\section{\textbf{Introduction}}

\label{sec:1}

Metric fixed point theory is important to find some applications in many
areas such as topology, analysis, differential equations etc. So different
generalizations of metric spaces were studied (see \cite{O Ege}, \cite{O Ege
2019}, \cite{ege}, \cite{gupta}, \cite{Mohanta}, \cite{Mustafa} and \cite%
{sedghi2}). For example, Mustafa and Sims introduced a new notion of
\textquotedblleft $G$-metric space\textquotedblright\ \cite{Mustafa}.
Mohanta proved some fixed point theorems for self-mappings satisfying some
kind of contractive type conditions on complete $G$-metric spaces \cite%
{Mohanta}.

Recently Sedghi, Shobe and Aliouche have defined the concept of $S$-metric
spaces in \cite{sedghi2} as follows:

\begin{definition}
\cite{sedghi2} \label{def1} Let $X$ be a nonempty set and $\mathcal{S}%
:X\times X\times X\rightarrow \lbrack 0,\infty )$ be a function satisfying
the following conditions for all $x,y,z,a\in X:$

$\mathbf{(S1)}$\textbf{\ }$\mathcal{S}(x,y,z)=0$ if and only if $x=y=z$,

$\mathbf{(S2)}$\textbf{\ }$\mathcal{S}(x,y,z)\leq \mathcal{S}(x,x,a)+%
\mathcal{S}(y,y,a)+\mathcal{S}(z,z,a)$.
\end{definition}

Then $\mathcal{S}$ is called an $S$-metric on $X$ and the pair $(X,\mathcal{S%
})$ is called an $S$-metric space.

Let $(X,d)$ be a complete metric space and $T$ be a self-mapping of $X$. In
\cite{rhoades}, $T$ is called a Rhoades' mapping if the following condition
is satisfied:%
\begin{equation*}
\mathbf{(R25)}\text{ \ \ \ \ }d(Tx,Ty)<\max
\{d(x,y),d(x,Tx),d(y,Ty),d(x,Ty),d(y,Tx)\},
\end{equation*}%
for each $x,y\in X$, $x\neq y.$ It was not given any fixed point result for
a Rhoades' mapping in \cite{rhoades}. Since then, many fixed point theorems
were obtained by several authors for a Rhoades' mapping (see \cite{chang},
\cite{sen} and \cite{liu}). Furthermore, the Rhoades' condition was extended
on $S$-metric spaces and it was presented new fixed point results (see \cite%
{nihal}, \cite{nihal2} and \cite{PhD}). Now we recall the Rhoades' condition
on an $S$-metric space.

Let $(X,\mathcal{S})$ be an $S$-metric space and $T$ be a self-mapping of $X$%
. In \cite{nihal} and \cite{PhD}, the present authors defined Rhoades'
condition $\mathbf{(S25)}$ on $(X,\mathcal{S})$ as follows$:$%
\begin{eqnarray*}
\mathbf{(S25)}\text{ \ \ \ \ }\mathcal{S}(Tx,Tx,Ty) &<&max\{\mathcal{S}%
(x,x,y),\mathcal{S}(Tx,Tx,x),\mathcal{S}(Ty,Ty,y), \\
&&\mathcal{S}(Ty,Ty,x),\mathcal{S}(Tx,Tx,y)\},
\end{eqnarray*}%
for each $x,y\in X$, $x\neq y$.

In this paper, to obtain a new generalization of the Rhoades' condition, we
introduce the notion of an $S$-normed space. We give some basic concepts and
topological definitions related to an $S$-norm. Then, we study a new form of
Rhoades' condition $\mathbf{(R25)}$ on $S$-normed spaces and obtain a fixed
point theorem. In Section \ref{sec:2}, we introduce the definition of an $S$%
-norm on $X$ and investigate some basic properties which are needed in the
sequel. We investigate the relationships among an $S$-norm and other known
concepts by counter examples. In Section \ref{sec:3}, we define Rhoades'
condition $(\mathbf{NS25})$ on an $S$-normed space. We study a fixed point
theorem using the condition $(\mathbf{NS25})$ and the notions of reflexive $%
S $-Banach space, $S$-normality, closure property and convexity. In Section %
\ref{sec:4}, we investigate some comparisons on $S$-normed spaces such as
the relationships between the conditions $(\mathbf{NR25})$ and $(\mathbf{NS25%
})$.

\section{$S$\textbf{-Normed Spaces}}

\label{sec:2} In this section, we introduce the notion of an $S$-normed
space and investigate some basic concepts related to an $S$-norm. We study
the relationships between an $S$-metric and an $S$-norm (resp. an $S$-norm
and a norm).

\begin{definition}
\label{def2} Let $X$ be a real vector space. A real valued function $\Vert
.,.,.\Vert :X\times X\times X\rightarrow
\mathbb{R}
$ is called an $S$-norm on $X$ if the following conditions hold:

$\mathbf{(NS1)}$ $\Vert x,y,z\Vert \geq 0$ and $\Vert x,y,z\Vert =0$ if and
only if $x=y=z=0$,

$\mathbf{(NS2)}$ $\Vert \lambda x,\lambda y,\lambda z\Vert =\mid \lambda
\mid \Vert x,y,z\Vert $ for all $\lambda \in
\mathbb{R}
$ and $x,y,z\in X$,

$\mathbf{(NS3)}$ $\Vert x+x^{\prime },y+y^{\prime },z+z^{\prime }\Vert \leq
\Vert 0,x,z^{\prime }\Vert +\Vert 0,y,x^{\prime }\Vert +\Vert 0,z,y^{\prime
}\Vert $ for all $x,y,z,x^{\prime },y^{\prime },z^{\prime }\in X$.

The pair $(X,\Vert .,.,.\Vert )$ is called an $S$-normed space.
\end{definition}

\begin{example}
\label{exm1} Let $X=%
\mathbb{R}
$ and $\Vert .,.,.\Vert :X\times X\times X\rightarrow
\mathbb{R}
$ be a function defined by%
\begin{equation*}
\Vert x,y,z\Vert =\mid x\mid +\mid y\mid +\mid z\mid ,
\end{equation*}%
for all $x,y,z\in X$. Then $(X,\Vert .,.,.\Vert )$ is an $S$-normed space.
Indeed, we show that the function $\Vert x,y,z\Vert =\mid x\mid +\mid y\mid
+\mid z\mid $ satisfies the conditions $(\mathbf{NS1})$, $(\mathbf{NS2})$
and $(\mathbf{NS3})$.

$(\mathbf{NS1})$ By the definition, clearly we have $\Vert x,y,z\Vert \geq 0$
for all $x,y,z\in X$. If $\Vert x,y,z\Vert =\mid x\mid +\mid y\mid +\mid
z\mid =0$, we obtain $x=y=z=0$.

$(\mathbf{NS2})$ Let $x,y,z\in X$ and $\lambda \in
\mathbb{R}
$. Then we have%
\begin{eqnarray*}
\Vert \lambda x,\lambda y,\lambda z\Vert &=&\mid \lambda x\mid +\mid \lambda
y\mid +\mid \lambda z\mid =\mid \lambda \mid \mid x\mid +\mid \lambda \mid
\mid y\mid +\mid \lambda \mid \mid z\mid \\
&=&\mid \lambda \mid (\mid x\mid +\mid y\mid +\mid z\mid )=\mid \lambda \mid
\Vert x,y,z\Vert .
\end{eqnarray*}
$(\mathbf{NS3})$ Let $x,y,z,x^{\prime },y^{\prime },z^{\prime }\in X$. Then
we obtain%
\begin{eqnarray*}
\Vert x+x^{\prime },y+y^{\prime },z+z^{\prime }\Vert &=&\mid x+x^{\prime
}\mid +\mid y+y^{\prime }\mid +\mid z+z^{\prime }\mid \\
&\leq &\mid x\mid +\mid x^{\prime }\mid +\mid y\mid +\mid y^{\prime }\mid
+\mid z\mid +\mid z^{\prime }\mid \\
&\leq &\mid 0\mid +\mid x\mid +\mid z^{\prime }\mid +\mid 0\mid +\mid y\mid
+\mid x^{\prime }\mid +\mid 0\mid +\mid z\mid +\mid y^{\prime }\mid \\
&=&\Vert 0,x,z^{\prime }\Vert +\Vert 0,y,x^{\prime }\Vert +\Vert
0,z,y^{\prime }\Vert .
\end{eqnarray*}%
Consequently, the function $\Vert x,y,z\Vert =\mid x\mid +\mid y\mid +\mid
z\mid $ satisfies the conditions $(\mathbf{NS1})$, $(\mathbf{NS2})$, $(%
\mathbf{NS3})$ and so $(X,\Vert .,.,.\Vert )$ is an $S$-normed space.
\end{example}

Now, we show that every $S$-norm generates an $S$-metric.

\begin{proposition}
\label{prop1} Let $(X,\Vert .,.,.\Vert )$ be an $S$-normed space. Then the
function $\mathcal{S}:X\times X\times X\rightarrow \lbrack 0,\infty )$
defined by%
\begin{equation}
S(x,y,z)=\Vert x-y,y-z,z-x\Vert  \label{eqn31}
\end{equation}%
is an $S$-metric on $X$.
\end{proposition}

\begin{proof}
Using the condition $(\mathbf{NS1})$, it can be easily seen that the
condition $(\mathbf{S1})$ is satisfied. We show that the condition $(\mathbf{%
S2})$ is satisfied. By $(\mathbf{NS3})$, we have%
\begin{eqnarray*}
S(x,y,z) &=&\Vert x-y,y-z,z-x\Vert \\
&=&\Vert x-a+a-y,y-a+a-z,z-a+a-x\Vert \\
&\leq &\Vert 0,x-a,a-x\Vert +\Vert 0,y-a,a-y\Vert +\Vert 0,z-a,a-z\Vert \\
&=&S(x,x,a)+S(y,y,a)+S(z,z,a),
\end{eqnarray*}%
for all $x,y,z,a\in X$.

Then, the function $\mathcal{S}$ is an $S$-metric and the pair $(X,S)$ is an
$S$-metric space.
\end{proof}

We call the $S$-metric defined in (\ref{eqn31}) as the $S$-metric generated
by the $S$-norm $\Vert .,.,.\Vert $ and denoted by $S_{\Vert .\Vert }$.

\begin{corollary}
\label{cor1} Every $S$-normed space is an $S$-metric space.
\end{corollary}

\begin{example}
\label{exm2} Let $X$ be a nonempty set, $(X,d)$ be a metric space and $%
\mathcal{S}:X\times X\times X\rightarrow \left[ 0,\infty \right) $ be the
function defined by%
\begin{equation*}
\mathcal{S}(x,y,z)=d(x,y)+d(x,z)+d(y,z),
\end{equation*}%
for all $x,y,z\in X$. Then the function is an $S$-metric on $X$ \cite%
{sedghi2}.

Let $X=%
\mathbb{R}
$. If we consider the usual metric $d$ on $X$, we obtain the $S$-metric
\begin{equation*}
\mathcal{S}(x,y,z)=\mid x-y\mid +\mid x-z\mid +\mid y-z\mid ,
\end{equation*}%
for all $x,y,z\in
\mathbb{R}
$. Using Proposition \ref{prop1}, we see that $\mathcal{S}$ is generated by
the $S$-norm defined in Example \ref{exm1}. Indeed we have
\begin{eqnarray*}
S(x,y,z) &=&\Vert x-y,y-z,z-x\Vert \\
&=&\mid x-y\mid +\mid y-z\mid +\mid z-x\mid \\
&=&\mid x-y\mid +\mid x-z\mid +\mid y-z\mid \\
&=&d(x,y)+d(x,z)+d(y,z),
\end{eqnarray*}%
for all $x,y,z\in X$.
\end{example}

\begin{lemma}
\label{lem5} An $S$-metric $\mathcal{S}$ generated by an $S$-norm on an $S$%
-normed space $X$ satisfies the following conditions

\begin{enumerate}
\item $\mathcal{S}(x+a,y+a,z+a)=\mathcal{S}(x,y,z)$,

\item $\mathcal{S}(\lambda x,\lambda y,\lambda z)=\left\vert \lambda
\right\vert \mathcal{S}(x,y,z)$,
\end{enumerate}

for $x,y,z,a\in X$ and every scalar $\lambda $.
\end{lemma}

\begin{proof}
The proof follows easily from the Proposition \ref{prop1}.
\end{proof}

We note that every $S$-metric can not be generated by an $S$-norm as we have
seen in the following example:

\begin{example}
\label{exm3} Let $X$ be a nonempty set and a function $S:X\times X\times
X\rightarrow \left[ 0,\infty \right) $ be defined by%
\begin{equation*}
S(x,y,z)=\left\{
\begin{array}{ccc}
0 & ; & \text{if }x=y=z \\
1 & ; & \text{otherwise}%
\end{array}%
,\right.
\end{equation*}%
for all $x,y,z\in X$. Then the function $\mathcal{S}$ is an $S$-metric on $X$%
. We call this $S$-metric is the discrete $S$-metric on $X$. The pair $(X,S)$
is called discrete $S$-metric space. Now, we prove that this $S$-metric can
not be generated by an $S$-norm. On the contrary, we assume that this $S$%
-metric is generated by an $S$-norm. Then the following equation should be
satisfied $:$%
\begin{equation*}
S(x,y,z)=\Vert x-y,y-z,z-x\Vert ,
\end{equation*}%
for all $x,y,z\in X$.

If we consider the case $x=y\neq z$ and $\mid \lambda \mid \neq 0,1$ then we
obtain%
\begin{eqnarray*}
S(\lambda x,\lambda y,\lambda z) &=&\Vert 0,\lambda (y-z),\lambda (z-x)\Vert
=1 \\
&\neq &\mid \lambda \mid S(x,y,z)=\mid \lambda \mid \Vert 0,y-z,z-x\Vert
=\mid \lambda \mid \text{,}
\end{eqnarray*}%
which is a contradiction with $(\mathbf{NS2})$. Consequently, this $S$%
-metric can not be generated by an $S$-norm.
\end{example}

We use the following result in the next section.

\begin{lemma}
\label{lem3} Let $(X,\Vert .,.,.\Vert )$ be an $S$-normed space. We have%
\begin{equation*}
\Vert 0,x-y,y-x\Vert =\Vert 0,y-x,x-y\Vert ,
\end{equation*}%
for each $x,y\in X$.
\end{lemma}

\begin{proof}
By the condition $(\mathbf{NS3})$, we get%
\begin{equation}
\Vert 0,x-y,y-x\Vert \leq \Vert 0,0,0\Vert +\Vert 0,0,0\Vert +\Vert
0,y-x,x-y\Vert =\Vert 0,y-x,x-y\Vert   \label{eqn29}
\end{equation}%
and%
\begin{equation}
\Vert 0,y-x,x-y\Vert \leq \Vert 0,0,0\Vert +\Vert 0,0,0\Vert +\Vert
0,x-y,y-x\Vert =\Vert 0,x-y,y-x\Vert .  \label{eqn30}
\end{equation}%
Using (\ref{eqn29}) and (\ref{eqn30}) we obtain $\Vert 0,x-y,y-x\Vert =\Vert
0,y-x,x-y\Vert $.
\end{proof}

We recall the definition of a norm on $X$ as follows.

Let $X$ be a real vector space. A real valued function $\left\Vert
.\right\Vert :X\rightarrow
\mathbb{R}
$ is called a norm on $X$ if the following conditions hold:

$(\mathbf{N1})$ $\Vert x\Vert \geq 0$ for all $x\in X$.

$(\mathbf{N2})$ $\Vert x\Vert =0$ if and only if $x=0$ for all $x\in X$.

$(\mathbf{N3})$ $\Vert \lambda x\Vert =\mid \lambda \mid \Vert x\Vert $ for
all $\lambda \in
\mathbb{R}
$ and $x\in X$.

$(\mathbf{N4})$ $\Vert x+y\Vert \leq \Vert x\Vert +\Vert y\Vert $ for all $%
x,y\in X$.

The pair $(X,\Vert .\Vert )$ is called a normed space.

We show that every norm generates an $S$-norm. We give the following
proposition.

\begin{proposition}
\label{prop2} Let $(X,\Vert .\Vert )$ be a normed space and a function $%
\Vert .,.,.\Vert :X\times X\times X\rightarrow
\mathbb{R}
$ be defined by%
\begin{equation}
\Vert x,y,z\Vert =\Vert x\Vert +\Vert y\Vert +\Vert z\Vert ,  \label{eqn32}
\end{equation}%
for all $x,y,z\in X$. Then $(X,\Vert .,.,.\Vert )$ is an $S$-normed space.
\end{proposition}

\begin{proof}
We show that the function $\Vert x,y,z\Vert =\Vert x\Vert +\Vert y\Vert
+\Vert z\Vert $ satisfies the conditions $(\mathbf{NS1})$, $(\mathbf{NS2})$
and $(\mathbf{NS3})$.

$(\mathbf{NS1})$ It is clear that $\Vert x,y,z\Vert \geq 0$ and $\Vert
x,y,z\Vert =0$ if and only if $x=y=z=0$.

$(\mathbf{NS2})$ Let $\lambda \in
\mathbb{R}
$ and $x,y,z\in X$. Then we obtain%
\begin{eqnarray*}
\Vert \lambda x,\lambda y,\lambda z\Vert &=&\Vert \lambda x\Vert +\Vert
\lambda y\Vert +\Vert \lambda z\Vert \\
&=&\mid \lambda \mid \Vert x\Vert +\mid \lambda \mid \Vert y\Vert +\mid
\lambda \mid \Vert z\Vert \\
&=&\mid \lambda \mid (\Vert x\Vert +\Vert y\Vert +\Vert z\Vert ) \\
&=&\mid \lambda \mid \Vert x,y,z\Vert .
\end{eqnarray*}%
$(\mathbf{NS3})$ Let $x,y,z,x^{\prime },y^{\prime },z^{\prime }\in X$. Then
we obtain%
\begin{eqnarray*}
\Vert x+x^{\prime },y+y^{\prime },z+z^{\prime }\Vert &=&\Vert x+x^{\prime
}\Vert +\Vert y+y^{\prime }\Vert +\Vert z+z^{\prime }\Vert \\
&\leq &\Vert x\Vert +\Vert x^{\prime }\Vert +\Vert y\Vert +\Vert y^{\prime
}\Vert +\Vert z\Vert +\Vert z^{\prime }\Vert \\
&=&\Vert 0\Vert +\Vert x\Vert +\Vert z^{\prime }\Vert +\Vert 0\Vert +\Vert
y\Vert +\Vert x^{\prime }\Vert +\Vert 0\Vert +\Vert z\Vert +\Vert y^{\prime
}\Vert \\
&=&\Vert 0,x,z^{\prime }\Vert +\Vert 0,y,x^{\prime }\Vert +\Vert
0,z,y^{\prime }\Vert .
\end{eqnarray*}%
Consequently, the function $\Vert x,y,z\Vert =\Vert x\Vert +\Vert y\Vert
+\Vert z\Vert $ satisfies the conditions $(\mathbf{NS1})$, $(\mathbf{NS2})$,
$(\mathbf{NS3})$ and so $(X,\Vert .,.,.\Vert )$ is an $S$-normed space.
\end{proof}

We have proved that every norm on $X$ defines an $S$-norm on $X$. We call
the $S$-norm defined in (\ref{eqn32}) as the $S$-norm generated by the norm $%
\Vert .\Vert $. For example, the $S$-norm defined in Example \ref{exm1} is
the $S$-norm generated by the usual norm on $%
\mathbb{R}
$.

There exists an $S$-norm which is not generated by a norm as we have seen in
the following example.

\begin{example}
\label{exm6} Let $X$ be a nonempty set and the function $\Vert .,.,.\Vert
:X\times X\times X\rightarrow
\mathbb{R}
$ be defined by%
\begin{equation*}
\Vert x,y,z\Vert =\mid x-2y-2z\mid +\mid y-2x-2z\mid +\mid z-2y-2x\mid \text{%
,}
\end{equation*}%
for all $x,y,z\in X$. Then, the function $\Vert x,y,z\Vert $ is an $S$-norm
on $X$, but it is not generated by a norm.

Now, we show that the conditions $(\mathbf{NS1})$, $(\mathbf{NS2})$ and $(%
\mathbf{NS3})$ are satisfied.

$(\mathbf{NS1})$ By the definition, clearly we obtain $\Vert x,y,z\Vert \geq
0$ and $\Vert x,y,z\Vert =0$ if and only if $x=y=z=0$ for all $x,y,z\in X$.

$(\mathbf{NS2})$ We have%
\begin{eqnarray*}
\Vert \lambda x,\lambda y,\lambda z\Vert &=&\mid \lambda x-2\lambda
y-2\lambda z\mid +\mid \lambda y-2\lambda x-2\lambda z\mid +\mid \lambda
z-2\lambda y-2\lambda x\mid \\
&=&\mid \lambda \mid (\mid x-2y-2z\mid +\mid y-2x-2z\mid +\mid z-2y-2x\mid )
\\
&=&\mid \lambda \mid \Vert x,y,z\Vert \text{,}
\end{eqnarray*}%
for all $\lambda \in
\mathbb{R}
$ and $x,y,z\in X$.

$(\mathbf{NS3})$ Let $x,y,z,x^{\prime },y^{\prime },z^{\prime }\in X$. Then
we obtain%
\begin{eqnarray*}
\Vert x+x^{\prime },y+y^{\prime },z+z^{\prime }\Vert &=&\mid x+x^{\prime
}-2y-2y^{\prime }-2z-2z^{\prime }\mid \\
+ &\mid &y+y^{\prime }-2x-2x^{\prime }-2z-2z^{\prime }\mid \\
+ &\mid &z+z^{\prime }-2y-2y^{\prime }-2x-2x^{\prime }\mid \\
&\leq &\mid 2x+2z^{\prime }\mid +\mid x-2z^{\prime }\mid +\mid z^{\prime
}-2x\mid \\
+ &\mid &2y+2x^{\prime }\mid +\mid y-2x^{\prime }\mid +\mid x^{\prime
}-2y\mid \\
+ &\mid &2z+2y^{\prime }\mid +\mid z-2y^{\prime }\mid +\mid y^{\prime
}-2z\mid \\
&=&\Vert 0,x,z^{\prime }\Vert +\Vert 0,y,x^{\prime }\Vert +\Vert
0,z,y^{\prime }\Vert \text{.}
\end{eqnarray*}%
Consequently, the function $\Vert x,y,z\Vert =\mid x-2y-2z\mid +\mid
y-2x-2z\mid +\mid z-2y-2x\mid $ is an $S$-norm on $X$.

On the contrary, we assume that this $S$-norm is generated by a norm. Then
the following equation should be satisfied%
\begin{equation*}
\Vert x,y,z\Vert =\Vert x\Vert +\Vert y\Vert +\Vert z\Vert \text{,}
\end{equation*}%
for all $x,y,z\in X$.

If we consider $\Vert x,0,0\Vert $ and $\Vert x,x,0\Vert $ then we obtain%
\begin{equation*}
\begin{array}{c}
\Vert x,0,0\Vert =\Vert x\Vert =\mid x\mid +\mid 2x\mid +\mid 2x\mid =5\mid
x\mid , \\
\Vert x,x,0\Vert =2\parallel x\parallel =\mid x\mid +\mid x\mid +\mid 4x\mid
=6\mid x\mid%
\end{array}%
\end{equation*}%
and so $\Vert x\Vert =5\mid x\mid $ and $\Vert x\Vert =3\mid x\mid $, which
is a contradiction. Hence this $S$-norm is not generated by a norm.
\end{example}

Now we prove that every $S$-norm generate a norm.

\begin{proposition}
\label{prop3} Let $X$ be a nonempty set, $(X,\Vert .,.,.\Vert )$ be an $S$%
-normed space and a function $\Vert .\Vert :X\rightarrow
\mathbb{R}
$ be defined as follows:%
\begin{equation*}
\Vert x\Vert =\Vert 0,x,0\Vert +\Vert 0,0,x\Vert ,
\end{equation*}%
for all $x\in X$. Then the function $\Vert x\Vert =\Vert 0,x,0\Vert +\Vert
0,0,x\Vert $ is a norm on $X$ and $(X,\Vert .\Vert )$ is a normed space.
\end{proposition}

\begin{proof}
Using the conditions $(\mathbf{NS1})$ and $(\mathbf{NS2})$, it is clear that
we obtain the conditions $(\mathbf{N1})$, $(\mathbf{N2})$ and $(\mathbf{N3})$
are satisfied.

Now we show that the condition $(\mathbf{N4})$ is satisfied.

$(\mathbf{N4})$ Let $x,y\in X$. By the condition $(\mathbf{NS3})$, we have%
\begin{eqnarray*}
\Vert x+y\Vert &=&\Vert 0,x+y,0\Vert +\Vert 0,0,x+y\Vert \\
&=&\Vert 0,x+y,0\Vert +\Vert 0,0,y+x\Vert \\
&\leq &\Vert 0,0,0\Vert +\Vert 0,x,0\Vert +\Vert 0,0,y\Vert +\Vert
0,0,x\Vert +\Vert 0,0,0\Vert +\Vert 0,y,0\Vert \\
&=&\Vert x\Vert +\Vert y\Vert .
\end{eqnarray*}%
Consequently, the function $\Vert x\Vert =\Vert 0,x,0\Vert +\Vert 0,0,x\Vert
$ is a norm on $X$ and $(X,\Vert .\Vert )$ is a normed space.
\end{proof}

We call this norm as the norm generated by the $S$-norm $\Vert .,.,.\Vert $.

Let $X$ be a real vector space. New generalizations of normed spaces have
been studied in recent years. For example, Khan defined the notion of a $G$%
-norm and studied some topological concepts in $G$-normed spaces \cite{khan}%
. Now we recall the definition of a $G$-norm and give the relationship
between a $G$-norm and an $S$-norm.

\begin{definition}
\cite{khan} \label{def5} Let $X$ be a real vector space. A real valued
function $\Vert .,.,.\Vert :X\times X\times X\rightarrow
\mathbb{R}
$ is called a $G$-norm on $X$ if the following conditions hold$:$

$(\mathbf{NG1})$ $\Vert x,y,z\Vert \geq 0$ and $\Vert x,y,z\Vert =0$ if and
only if $x=y=z=0$.

$(\mathbf{NG2})$ $\Vert x,y,z\Vert $ is invariant under permutations of $%
x,y,z$.

$(\mathbf{NG3})$ $\Vert \lambda x,\lambda y,\lambda z\Vert =\mid \lambda
\mid \Vert x,y,z\Vert $ for all $\lambda \in
\mathbb{R}
$ and $x,y,z\in X$.

$(\mathbf{NG4})$ $\Vert x+x^{\prime },y+y^{\prime },z+z^{\prime }\Vert \leq
\Vert x,y,z\Vert +\Vert x^{\prime },y^{\prime },z^{\prime }\Vert $ for all $%
x,y,z,x^{\prime },y^{\prime },z^{\prime }\in X$.

$(\mathbf{NG5})$ $\Vert x,y,z\Vert \geq \Vert x+y,0,z\Vert $ for all $%
x,y,z\in X$.

The pair $(X,\Vert .,.,.\Vert )$ is called a $G$-normed space.
\end{definition}

\begin{proposition}
\label{prop5} Every $G$-normed space is an $S$-normed space.
\end{proposition}

\begin{proof}
Using the conditions $(\mathbf{NG1})$ and $(\mathbf{NG3})$, we see that the
conditions $(\mathbf{NS1})$ and $(\mathbf{NS2})$ are satisfied. We only show
that the condition $(\mathbf{NS3})$ is satisfied.

$(\mathbf{NS3})$ Let $x,y,z,x^{\prime },y^{\prime },z^{\prime }\in X$. Using
the conditions $(\mathbf{NG2})$ and $(\mathbf{NG4})$, we obtain%
\begin{eqnarray*}
\Vert x+x^{\prime },y+y^{\prime },z+z^{\prime }\Vert &=&\Vert
(x+0)+x^{\prime },0+(y+y^{\prime }),z^{\prime }+z\Vert \\
&\leq &\Vert x+0,0,0+z^{\prime }\Vert +\Vert x^{\prime },y+y^{\prime },z\Vert
\\
&=&\Vert 0,0+x,0+z^{\prime }\Vert +\Vert x^{\prime },y+y^{\prime },z\Vert \\
&\leq &\Vert 0,0,0\Vert +\Vert 0,x,z^{\prime }\Vert +\Vert x^{\prime
},y,0\Vert +\Vert 0,y^{\prime },z\Vert \\
&=&\Vert 0,x,z^{\prime }\Vert +\Vert 0,y,x^{\prime }\Vert +\Vert
0,z,y^{\prime }\Vert \text{.}
\end{eqnarray*}%
Consequently, the condition $(\mathbf{NS3})$ is satisfied.
\end{proof}

The converse of Proposition \ref{prop5} can not be always true as we have
seen in the following example.

\begin{example}
\label{exm8} Let $X=%
\mathbb{R}
$ and the $S$-norm be defined as in Example \ref{exm6}. If we put $x=1$, $%
y=5 $ and $z=0$, the condition $(\mathbf{NG5})$ is not satisfied. Indeed, we
have%
\begin{equation*}
\Vert x,y,z\Vert =\mid x-2y-2z\mid +\mid y-2x-2z\mid +\mid z-2y-2x\mid =23
\end{equation*}%
and%
\begin{equation*}
\Vert x+y,0,z\Vert =\mid x+y-2z\mid +\mid 2x+2y+2z\mid +\mid z-2y-2x\mid =30.
\end{equation*}%
Hence this $S$-norm is not a $G$-norm on $%
\mathbb{R}
$.
\end{example}

Now we give the definitions of an open ball and a closed ball on\ an $S$%
-normed space.

\begin{definition}
\label{def3} Let $(X,\Vert .,.,.\Vert )$ be an $S$-normed space. For given $%
x_{0}$, $a_{1}$, $a_{2}\in X$ and $r>0$, the open ball $%
B_{a_{1}}^{a_{2}}(x_{0},r)$ and the closed ball $B_{a_{1}}^{a_{2}}[x_{0},r]$
are defined as follows:%
\begin{equation*}
B_{a_{1}}^{a_{2}}(x_{0},r)=\{y\in X:\Vert y-x_{0},y-a_{1},y-a_{2}\Vert <r\}
\end{equation*}%
and%
\begin{equation*}
B_{a_{1}}^{a_{2}}[x_{0},r]=\{y\in X:\Vert y-x_{0},y-a_{1},y-a_{2}\Vert \leq
r\}\text{.}
\end{equation*}
\end{definition}

\begin{example}
\label{exm4} Let us consider the $S$-normed space $(X,\Vert .,.,.\Vert )$
generated by the usual norm on $X$, where $X=%
\mathbb{R}
^{2}$ and
\begin{equation*}
\Vert x\Vert =\Vert (x_{1},x_{2})\Vert =\sqrt{x_{1}^{2}+x_{2}^{2}}\text{,}
\end{equation*}%
for all $x\in
\mathbb{R}
^{2}$. Then the open ball $B_{a_{1}}^{a_{2}}(x_{0},r)$ in $%
\mathbb{R}
^{2}$ is a $3$-ellipse given by%
\begin{equation*}
B_{a_{1}}^{a_{2}}(x_{0},r)=\{y\in
\mathbb{R}
^{2}:\Vert y-x_{0}\Vert +\Vert y-a_{1}\Vert +\Vert y-a_{2}\Vert <r\}\text{.}
\end{equation*}%
If we choose $y=(y_{1},y_{2})$, $x_{0}=(1,1)$, $a_{1}=(0,0)$, $a_{2}=(-1,-1)$
in $%
\mathbb{R}
^{2}$ and $r=5$. Then we obtain%
\begin{equation}
B_{a_{1}}^{a_{2}}(x_{0},r)=\{y\in
\mathbb{R}
^{2}:\sqrt{(y_{1}-1)^{2}+(y_{2}-1)^{2}}+\sqrt{y_{1}^{2}+y_{2}^{2}}+\sqrt{%
(y_{1}+1)^{2}+(y_{2}+1)^{2}}<5\}\text{,}  \label{eqn33}
\end{equation}%
as shown in Figure \ref{fig:1A}.
\end{example}

\begin{figure}[h]
\centering
\begin{subfigure}{.5\textwidth}
  \centering
  \includegraphics[width=.4\linewidth]{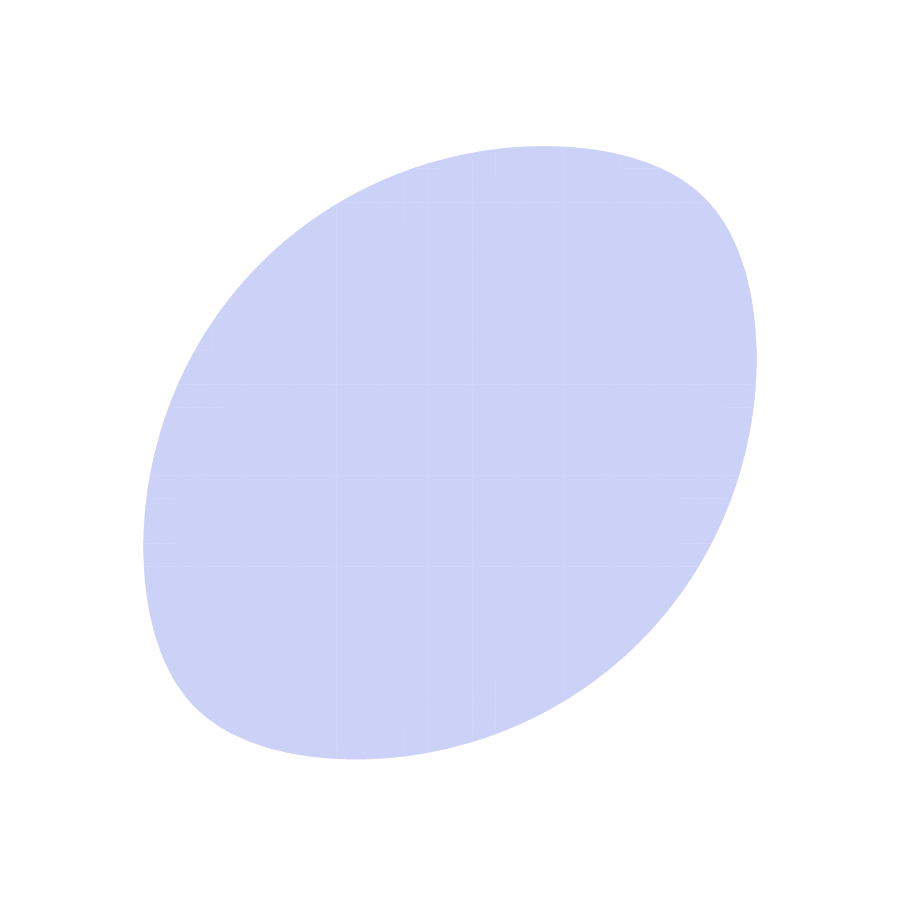}
  \caption{\Small The open ball which is corresponding\\ to the $S$-norm defined in (\ref{eqn33}).}
  \label{fig:1A}
\end{subfigure}%
\begin{subfigure}{.5\textwidth}
  \centering
  \includegraphics[width=.4\linewidth]{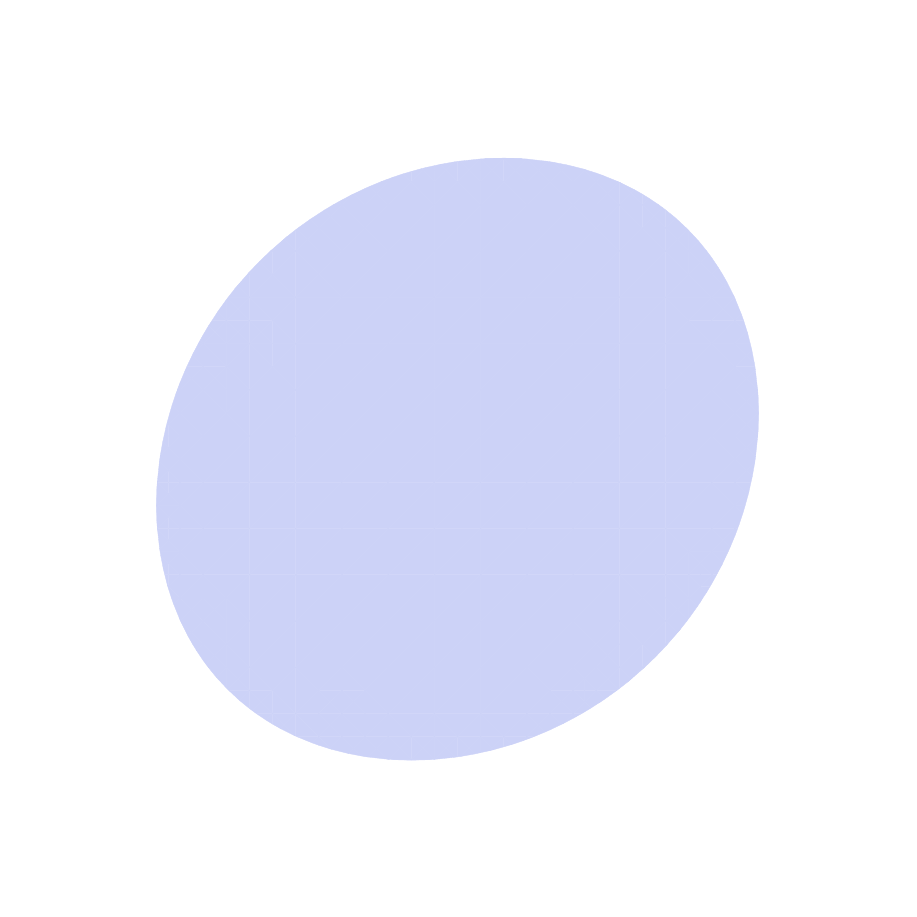}
  \caption{\Small  The open ball which is corresponding\\ to the $S$-norm defined in (\ref{eqn34}).}
  \label{fig:1B}
\end{subfigure}
\caption{\Small Some open balls in $(\mathbb{R}^{2},\Vert .,.,.\Vert)$}
\label{fig:1}
\end{figure}

Now we give the following example using the $S$-norm which is not generated
by a norm.

\begin{example}
\label{exm7} Let $X=%
\mathbb{R}
^{2}$ and the function $\Vert .,.,.\Vert :X\times X\times X\rightarrow
\mathbb{R}
$ be defined as in Example \ref{exm6}. Then we have%
\begin{eqnarray*}
\Vert x,y,z\Vert &=&\mid x-2y-2z\mid +\mid y-2x-2z\mid +\mid z-2y-2x\mid \\
&=&\sqrt{(x_{1}-2y_{1}-2z_{1})^{2}+(x_{2}-2y_{2}-2z_{2})^{2}} \\
&&+\sqrt{(y_{1}-2x_{1}-2z_{1})^{2}+(y_{2}-2x_{2}-2z_{2})^{2}} \\
&&+\sqrt{(z_{1}-2y_{1}-2x_{1})^{2}+(z_{2}-2y_{2}-2x_{2})^{2}},
\end{eqnarray*}%
for all $x=(x_{1},x_{2})$, $y=(y_{1},y_{2})$, $z=(z_{1},z_{2})\in
\mathbb{R}
^{2}$. Then $(%
\mathbb{R}
^{2},\Vert .,.,.\Vert )$ is an $S$-normed space. The open ball $%
B_{a_{1}}^{a_{2}}(x_{0},r)$ in $%
\mathbb{R}
^{2}$ is%
\begin{equation*}
B_{a_{1}}^{a_{2}}(x_{0},r)=\{y\in
\mathbb{R}
^{2}:\Vert y-x_{0},y-a_{1},y-a_{2}\Vert <r\}\text{.}
\end{equation*}%
If we choose $y=(y_{1},y_{2})$, $x_{0}=(1,1)$, $a_{1}=(0,0)$, $a_{2}=(-1,-1)$
in $%
\mathbb{R}
^{2}$ and $r=20$. Then we obtain%
\begin{equation}
\begin{array}{c}
B_{a_{1}}^{a_{2}}(x_{0},r)=\{y\in
\mathbb{R}
^{2}:\sqrt{(3y_{1}+3)^{2}+(3y_{2}+3)^{2}}+\sqrt{9y_{1}^{2}+9y_{2}^{2}} \\
+\sqrt{(3-3y_{1})^{2}+(3-3y_{2})^{2}}<20\}\text{,}%
\end{array}
\label{eqn34}
\end{equation}%
as shown in Figure \ref{fig:1B}.
\end{example}

\begin{definition}
\label{def4} Let $(X,\Vert .,.,.\Vert )$ be an $S$-normed space.

\begin{enumerate}
\item A sequence $\{x_{n}\}$ in $X$ converges to $x$ if and only if%
\begin{equation*}
\underset{n\rightarrow \infty }{\lim }\Vert 0,x_{n}-x,x-x_{n}\Vert =0.
\end{equation*}
That is, for each $\varepsilon >0$ there exists $n_{0}\in
\mathbb{N}
$ such that%
\begin{equation*}
\Vert 0,x_{n}-x,x-x_{n}\Vert <\varepsilon \text{,}
\end{equation*}%
for all $n\geq n_{0}$.

\item A sequence $\{x_{n}\}$ in $X$ is called a Cauchy sequence if
\begin{equation*}
\underset{n,m,l\rightarrow \infty }{\lim }\Vert
x_{n}-x_{m},x_{m}-x_{l},x_{l}-x_{n}\Vert =0\text{.}
\end{equation*}%
That is, for each $\varepsilon >0$ there exists $n_{0}\in
\mathbb{N}
$ such that%
\begin{equation*}
\Vert x_{n}-x_{m},x_{m}-x_{l},x_{l}-x_{n}\Vert <\varepsilon \text{,}
\end{equation*}%
for all $n,m,l\geq n_{0}$.

\item An $S$-normed space is called complete if each Cauchy sequence in $X$
converges in $X$.

\item A complete $S$-normed space is called an $S$-Banach space.
\end{enumerate}
\end{definition}

\begin{proposition}
\label{prop4} Every convergent sequence in an $S$-normed space is a Cauchy
sequence.
\end{proposition}

\begin{proof}
Let the sequence $\{x_{n}\}$ in $X$ be convergent to $x$. For each $%
\varepsilon >0$, there exists $n_{0}\in
\mathbb{N}
$ such that%
\begin{equation*}
\Vert 0,x_{n}-x,x-x_{n}\Vert <\frac{\varepsilon }{3}\text{,}
\end{equation*}%
for all $n\geq n_{0}$. We now show that for each $\varepsilon >0$ there
exists $n_{0}\in
\mathbb{N}
$ such that%
\begin{equation*}
\Vert x_{n}-x_{m},x_{m}-x_{l},x_{l}-x_{n}\Vert <\varepsilon \text{,}
\end{equation*}%
for all $n,m,l\geq n_{0}$. Using the condition $(\mathbf{NS3})$, we obtain%
\begin{eqnarray*}
\Vert x_{n}-x_{m},x_{m}-x_{l},x_{l}-x_{n}\Vert &=&\Vert
x_{n}-x+x-x_{m},x_{m}-x+x-x_{l},x_{l}-x+x-x_{n}\Vert \\
&\leq &\Vert 0,x_{n}-x,x-x_{n}\Vert +\Vert 0,x_{m}-x,x-x_{m}\Vert \\
&&+\Vert 0,x_{l}-x,x-x_{l}\Vert \\
&<&\frac{\varepsilon }{3}+\frac{\varepsilon }{3}+\frac{\varepsilon }{3}%
=\varepsilon \text{.}
\end{eqnarray*}%
Consequently, the sequence $\{x_{n}\}$ in $X$ is a Cauchy sequence.
\end{proof}

The converse of Proposition \ref{prop4} can not be always true as we have
seen in the following example.

\begin{example}
\label{exm5} Let $X=(0,1)\subset
\mathbb{R}
$ and the function $\Vert .,.,.\Vert :X\times X\times X\rightarrow
\mathbb{R}
$ be an $S$-norm generated by the usual norm on $X$. If we consider the
sequence $\{x_{n}\}=\left\{ \dfrac{1}{n}\right\} $ on $X$, then this
sequence is a Cauchy sequence, but it is not a convergent sequence on $X$.

Now we show that the sequence is a Cauchy sequence. For $x_{n}$, $x_{m}$, $%
x_{l}\in X$, we obtain%
\begin{eqnarray*}
\underset{n,m,l\rightarrow \infty }{\lim }\Vert
x_{n}-x_{m},x_{m}-x_{l},x_{l}-x_{n}\Vert &=&\underset{n,m,l\rightarrow
\infty }{\lim }\left\Vert \frac{1}{n}-\frac{1}{m},\frac{1}{m}-\frac{1}{l},%
\frac{1}{l}-\frac{1}{n}\right\Vert \\
&=&\underset{n,m,l\rightarrow \infty }{\lim }\left( \left\vert \frac{1}{n}-%
\frac{1}{m}\right\vert +\left\vert \frac{1}{m}-\frac{1}{l}\right\vert
+\left\vert \frac{1}{l}-\frac{1}{n}\right\vert \right) =0\text{.}
\end{eqnarray*}%
The sequence is convergent to $0$ as follows:%
\begin{equation*}
\underset{n\rightarrow \infty }{\lim }\Vert 0,x_{n}-x,x-x_{n}\Vert =\underset%
{n\rightarrow \infty }{\lim }\left\Vert 0,\frac{1}{n}-0,0-\frac{1}{n}%
\right\Vert =0\text{,}
\end{equation*}%
for all $x_{n}\in X$. But $0\notin X$. Consequently, the sequence is not
convergent on $X$.
\end{example}

\section{A Fixed Point Theorem on $S$-Normed Spaces}

\label{sec:3} In this section, we introduce the Rhoades' condition on an $S$%
-normed space and denote it by $(\mathbf{NS25})$. We prove a fixed point
teorem using this contractive condition.

At first, we give some definitions and a proposition which are needed in the
sequel.

\begin{definition}
\label{def6} Let $(X,\Vert .,.,.\Vert )$ be an $S$-normed space and $%
E\subseteq X$. The closure of $E$, denoted by $\overline{E}$, is the set of
all $x\in X$ such that there exists a sequence $\{x_{n}\}$ in $E$ converging
to $x$. If $E=\overline{E}$, then $E$ is called a closed set.
\end{definition}

\begin{definition}
\label{def7} Let $(X,\Vert .,.,.\Vert )$ be an $S$-normed space and $%
A\subseteq X$. The subset $A$ is called bounded if there exists $r>0$ such
that%
\begin{equation*}
\Vert 0,x-y,y-x\Vert <r,
\end{equation*}%
for all $x,y\in A$.
\end{definition}

\begin{definition}
\label{def8} Let $(X,\Vert .,.,.\Vert )$ be an $S$-normed space and $%
A\subseteq X$. The $S$-diameter of $A$ is defined by%
\begin{equation*}
\delta ^{s}(A)=\sup \{\Vert 0,x-y,y-x\Vert :x,y\in A\}\text{.}
\end{equation*}%
If $A$ is bounded then we will write $\delta ^{s}(A)<\infty $.
\end{definition}

\begin{definition}
\label{def9} Let $X$ be an $S$-Banach space, $A\subseteq X$ and $u\in X$.

\begin{enumerate}
\item The $S$-radius of $A$ relative to a given $u\in X$ is defined by%
\begin{equation*}
r_{u}^{s}(A)=\sup \{\Vert 0,u-x,x-u\Vert :x\in A\}\text{.}
\end{equation*}

\item The $S$-Chebyshev radius of $A$ is defined by%
\begin{equation*}
r^{s}(A)=\inf \{r_{u}^{s}(A):u\in A\}\text{.}
\end{equation*}

\item The $S$-Chebyshev centre of $A$ is defined by%
\begin{equation*}
C^{s}(A)=\{u\in A:r_{u}^{s}(A)=r^{s}(A)\}\text{.}
\end{equation*}
\end{enumerate}
\end{definition}

By Definition \ref{def8} and Definition \ref{def9}, it can be easily seen
the following inequality:

\begin{equation*}
r^{s}(A)\leq r_{u}^{s}(A)\leq \delta ^{s}(A)\text{.}
\end{equation*}

\begin{definition}
\label{def10} A point $u\in A$ is called $S$-diametral if $%
r_{u}^{s}(A)=\delta ^{s}(A)$. If $r_{u}^{s}(A)<\delta ^{s}(A)$, then $u$ is
called non-$S$-diametral.
\end{definition}

\begin{definition}
\label{def11} A convex subset of an $S$-Banach space $X$ has $S$-normal
structure if every $S$-bounded and convex subset of $A$ having $\delta
^{s}(A)>0$ has at least one non-$S$-diametral point.
\end{definition}

\begin{proposition}
\label{prop6} If $X$ is a reflexive $S$-Banach space, $A$ is a nonempty,
closed and convex subset of $X$, then $C^{s}(A)$ is nonempty, closed and
convex.
\end{proposition}

\begin{proof}
It can be easily seen by definition of $C^{s}(A)$.
\end{proof}

Now we introduce the Rhoades' condition $(\mathbf{NS25})$ on an $S$-Banach
space.

\begin{definition}
\label{def12} Let $(X,\Vert .,.,.\Vert )$ be an $S$-Banach space and $T$ be
a self-mapping of $X$. We define%
\begin{equation*}
\begin{array}{r}
(\mathbf{NS25})\text{ \ \ }\Vert 0,Tx-Ty,Ty-Tx\Vert <\max \{\Vert
0,x-y,y-x\Vert ,\Vert 0,Tx-x,x-Tx\Vert , \\
\Vert 0,Ty-y,y-Ty\Vert ,\Vert 0,Ty-x,x-Ty\Vert ,\Vert 0,Tx-y,y-Tx\Vert \}%
\text{,}%
\end{array}%
\end{equation*}%
for each $x,y\in X$, $x\neq y$.
\end{definition}

\begin{lemma}
\cite{smulian} \label{lem1} Let $X$ be a Banach space. Then $X$ is reflexive
if and only if for any decreasing sequence $\{K_{n}\}$ of nonempty, bounded,
closed and convex subsets of $X$,%
\begin{equation*}
\bigcap\limits_{n=1}^{\infty }K_{n}\neq \emptyset \text{.}
\end{equation*}
\end{lemma}

\begin{lemma}
\label{lem4} Let $X$ be an $S$-Banach space. Then $X$ is reflexive if and
only if for any decreasing sequence $\{K_{n}\}$ of nonempty, bounded, closed
and convex subsets of $X$,%
\begin{equation*}
\bigcap\limits_{n=1}^{\infty }K_{n}\neq \emptyset \text{.}
\end{equation*}
\end{lemma}

\begin{proof}
By the definition of reflexivity the proof follows easily.
\end{proof}

Recall that the convex hull of a set $A$ is denoted by $conv(A)$ and any
member of this set $conv(A)$ has the form%
\begin{equation*}
\overset{n}{\sum\limits_{i=1}}\alpha _{i}x_{i}\text{,}
\end{equation*}%
where $x_{i}\in A_{i},\alpha _{i}\geq 0$ for all $i=1,...,n$ and $\overset{n}%
{\sum\limits_{i=1}}\alpha _{i}=1$.

Now, we give the following fixed point theorem.

\begin{theorem}
\label{thm1} Let $X$ be a reflexive $S$-Banach space and $A$ be a nonempty,
closed, bounded and convex subset of $X$, having $S$-normal structure. If $%
T:A\rightarrow A$ is a continuous self-mapping satisfying the condition $(%
\mathbf{NS25})$ then $T$ has a unique fixed point in $A$.
\end{theorem}

\begin{proof}
At first, we show that the existence of the fixed point. Let $\mathcal{A}$
be the family of every nonempty, closed and convex subsets of $A$. Also we
assume that if $F\in $ $\mathcal{A}$ then $TF\subseteq F$. The family $%
\mathcal{A}$ is nonempty since $A\in $ $\mathcal{A}$. We can partially order
$\mathcal{A}$ by set inclusion, that is, if $F_{1}\subseteq F_{2}$ then $%
F_{1}\leq F_{2}$.

In $\mathcal{A}$, if we define a decreasing net of subsets%
\begin{equation*}
S=\{F_{i}:F_{i}\in \text{$\mathcal{A}$},i\in I\}\text{,}
\end{equation*}%
then by reflexivity, this net $S$ has nonempty intersection. Because it is a
decreasing net of nonempty, closed, bounded and convex subsets of $X$. If we
put $F_{0}=\bigcap\limits_{i\in I}F_{i}$ we have that $F_{0}$ is in $%
\mathcal{A}$ and is a lower bound of $S$.

Using Zorn's Lemma, there is a minimal element, denoted by $F$, in $\mathcal{%
A}$ as $S$ is any arbitrary decreasing net in $\mathcal{A}$. We see that
this $F$ is a singleton.

Assume that $\delta ^{s}(F)\neq \emptyset $. Since $F$ is nonempty, closed
and convex, $C^{s}(F)$ is a nonempty, closed and convex subset of $F$. We
have that%
\begin{equation*}
r^{s}(F)<\delta ^{s}(F)\text{,}
\end{equation*}%
\begin{equation*}
\delta ^{s}(C^{s}(F))\leq r^{s}(F)<\delta ^{s}(F)
\end{equation*}%
and so $C^{s}(F)$ is a proper subset of $F$.

Let $(F_{m})_{m\in
\mathbb{N}
}$ be an increasing sequence of subsets of $F$, defined by%
\begin{equation*}
F_{1}=C^{s}(F)\text{ and }F_{m+1}=conv(F_{m}\cup TF_{m})\text{,}
\end{equation*}%
for all $m\in
\mathbb{N}
$. If we denote the $S$-diameters of these sets $F_{k}$ by $\delta
_{k}^{s}=\delta ^{s}(F_{k})$, we show that%
\begin{equation*}
\delta _{k}^{s}\leq r^{s}(F)\text{,}
\end{equation*}%
for all $k\in
\mathbb{N}
$.

Using the $(PMI)$, we obtain

\begin{enumerate}
\item For $k=1$,%
\begin{equation*}
\delta _{1}^{s}=\delta ^{s}(F_{1})=\delta ^{s}(C^{s}(F))\leq r^{s}(F)\text{.}
\end{equation*}

\item If $\delta _{k}^{s}\leq r^{s}(F)$ for every $k=1,...,m$ then $\delta
_{m+1}^{s}\leq r^{s}(F)$.
\end{enumerate}

We note that%
\begin{equation*}
\delta _{m+1}^{s}=\delta ^{s}(F_{m+1})=\delta ^{s}(conv(F_{m}\cup
TF_{m}))=\delta ^{s}(F_{m}\cup TF_{m})\text{.}
\end{equation*}%
By the definition of $S$-diameter, for any given $\varepsilon >0$ there are $%
x^{\prime }$ and $y^{\prime }$ in $F_{m}\cup T(F_{m})$ satisfying%
\begin{equation*}
\delta _{m+1}^{s}-\varepsilon <\Vert 0,x^{\prime }-y^{\prime },y^{\prime
}-x^{\prime }\Vert \leq \delta _{m+1}^{s}\text{.}
\end{equation*}%
We obtain the following three cases for $x^{\prime }$, $y^{\prime }$:

\begin{enumerate}
\item $x^{\prime }$, $y^{\prime }\in F_{m}$ or

\item $x^{\prime }\in F_{m}$ and $y^{\prime }\in TF_{m}$ or

\item $x^{\prime }$, $y^{\prime }\in TF_{m}$.
\end{enumerate}

Redefining $x^{\prime }$ and $y^{\prime }$ as follows:

\begin{enumerate}
\item $x^{\prime }=x$ and $y^{\prime }=y$ with $x,y\in F_{m}$,

\item $x^{\prime }=x$ and $y^{\prime }=Ty$ with $x,y\in F_{m}$,

\item $x^{\prime }=Tx$ and $y^{\prime }=Ty$ with $x,y\in F_{m}$.
\end{enumerate}

We show that in any case%
\begin{equation*}
\delta _{m+1}^{s}-\varepsilon <r^{s}(F)\text{.}
\end{equation*}

\textbf{Case 1}. By the definition of $\delta _{m}^{s}$ and the induction
hypothesis, we obtain%
\begin{equation}
\delta _{m+1}^{s}-\varepsilon <\Vert 0,x-y,y-x\Vert \leq \delta
_{m+1}^{s}\leq r^{s}(F)  \label{eqn1}
\end{equation}%
and so $\delta _{m+1}^{s}-\varepsilon <r^{s}(F)$.

\textbf{Case 2}. We obtain%
\begin{equation*}
\delta _{m+1}^{s}-\varepsilon <\Vert 0,x-Ty,Ty-x\Vert
\end{equation*}%
with $x$, $y\in F_{m}$. Then by the definition of $F_{m}$, we have $x,y\in
conv(F_{m-1}\cup TF_{m-1})$ and so there is a finite index set $I$ such that
$x=\sum\limits_{i\in I}\alpha _{i}x_{i}$, with $\sum\limits_{i\in I}\alpha
_{i}=1$, $\alpha _{i}\geq 0$ and $x_{i}\in F_{m-1}\cup TF_{m-1}$ for any $%
i\in I$. We can separate the set $I$ in two disjoint subsets, $I=I_{1}\cup
I_{2}$, such that if $i\in I_{1}$ then $x_{i}\in F_{m-1}$ and if $i\in I_{2}$
then $x_{i}\in TF_{m-1}$.

Now redefining $x_{i}$ as $x_{i}=Tx_{i}$ with $x_{i}\in F_{m-1}$, we obtain%
\begin{equation*}
x=\sum\limits_{i\in I_{1}}\alpha _{i}x_{i}+\sum\limits_{i\in I_{2}}\alpha
_{i}Tx_{i}\text{.}
\end{equation*}%
Substituting in $\Vert 0,x-Ty,Ty-x\Vert $, we get%
\begin{equation}
\Vert 0,x-Ty,Ty-x\Vert \leq \sum\limits_{i\in I_{1}}\alpha _{i}\Vert
0,x_{i}-Ty,Ty-x_{i}\Vert +\sum\limits_{i\in I_{2}}\alpha _{i}\Vert
0,Tx_{i}-Ty,Ty-Tx_{i}\Vert \text{.}  \label{eqn2}
\end{equation}%
Applying the condition $(\mathbf{NS25})$ to $\Vert
0,Tx_{i}-Ty,Ty-Tx_{i}\Vert $, we have%
\begin{equation}
\begin{array}{r}
\Vert 0,Tx_{i}-Ty,Ty-Tx_{i}\Vert <\max \{\Vert 0,x_{i}-y,y-x_{i}\Vert ,\Vert
0,x_{i}-Tx_{i},Tx_{i}-x_{i}\Vert , \\
\Vert 0,y-Ty,Ty-y\Vert ,\Vert 0,x_{i}-Ty,Ty-x_{i}\Vert ,\Vert
0,Tx_{i}-y,y-Tx_{i}\Vert \}\text{.}%
\end{array}
\label{eqn3}
\end{equation}%
As $x_{i}\in F_{m-1}$, $Tx_{i},y\in F_{m}$, we have%
\begin{equation*}
\begin{array}{l}
\Vert 0,x_{i}-y,y-x_{i}\Vert \leq r^{s}(F), \\
\Vert 0,x_{i}-Tx_{i},Tx_{i}-x_{i}\Vert \leq r^{s}(F), \\
\Vert 0,Tx_{i}-y,y-Tx_{i}\Vert \leq r^{s}(F)%
\end{array}%
\end{equation*}%
and replacing in (\ref{eqn3}), we obtain%
\begin{equation*}
\Vert 0,Tx_{i}-Ty,Ty-Tx_{i}\Vert <\max \{r^{s}(F),\Vert 0,y-Ty,Ty-y\Vert
,\Vert 0,x_{i}-Ty,Ty-x_{i}\Vert \}\text{.}
\end{equation*}%
Let us subdivide the index set $I_{2}$ in three disjoint subsets $%
I_{2}=I_{2}^{1}\cup I_{2}^{2}\cup I_{2}^{3}$ such that%
\begin{equation*}
\begin{array}{l}
I_{2}^{1}=\{i\in I_{2}:\Vert 0,Tx_{i}-Ty,Ty-Tx_{i}\Vert <r^{s}(F)\}, \\
I_{2}^{2}=\{i\in I_{2}:\Vert 0,Tx_{i}-Ty,Ty-Tx_{i}\Vert <\Vert
0,x_{i}-Ty,Ty-x_{i}\Vert \}, \\
I_{2}^{3}=\{i\in I_{2}:\Vert 0,Tx_{i}-Ty,Ty-Tx_{i}\Vert <\Vert
0,y-Ty,Ty-y\Vert \}.%
\end{array}%
\end{equation*}%
Then using (\ref{eqn2}), we have
\begin{equation}
\begin{array}{c}
\Vert 0,x-Ty,Ty-x\Vert \leq \sum\limits_{i\in I_{1}\cup I_{2}^{2}}\alpha
_{i}\Vert 0,x_{i}-Ty,Ty-x_{i}\Vert +\sum\limits_{i\in I_{2}^{1}}\alpha
_{i}r^{s}(F) \\
+\sum\limits_{i\in I_{2}^{3}}\alpha _{i}\Vert 0,y-Ty,Ty-y\Vert \text{.}%
\end{array}
\label{eqn4}
\end{equation}%
Redefining $I_{1}$, $I_{2}$ and $I_{3}$%
\begin{equation}
\overline{I_{1}}=I_{1}\cup I_{2}^{2}\text{, }\overline{I_{2}}=I_{2}^{1}\text{
and }\overline{I_{3}}=I_{2}^{3}\text{.}  \label{eqn5}
\end{equation}%
Then we have $I=\overline{I_{1}}\cup \overline{I_{2}}\cup \overline{I_{3}}$,
with $\overline{I_{j}}\cap \overline{I_{k}}=\emptyset $. If $j\neq k$ and $%
\sum\limits_{i\in I}\alpha _{i}=1$ then using (\ref{eqn4}), it becomes%
\begin{equation}
\begin{array}{c}
\Vert 0,x-Ty,Ty-x\Vert \leq \sum\limits_{i\in \overline{I_{1}}}\alpha
_{i}\Vert 0,x_{i}-Ty,Ty-x_{i}\Vert +\sum\limits_{i\in \overline{I_{2}}%
}\alpha _{i}r^{s}(F) \\
+\sum\limits_{i\in \overline{I_{3}}}\alpha _{i}\Vert 0,y-Ty,Ty-y\Vert \text{.%
}%
\end{array}
\label{eqn6}
\end{equation}%
If $A_{0}=\sum\limits_{i\in \overline{I_{2}}}\alpha _{i}$ and $%
B_{0}=\sum\limits_{i\in \overline{I_{3}}}\alpha _{i}$ with $%
\sum\limits_{i\in \overline{I_{1}}}\alpha _{i}+A_{0}+B_{0}=1$. Using (\ref%
{eqn6}), we have%
\begin{equation}
\Vert 0,x-Ty,Ty-x\Vert \leq \sum\limits_{i\in \overline{I_{1}}}\alpha
_{i}\Vert 0,x_{i}-Ty,Ty-x_{i}\Vert +A_{0}r^{s}(F)+B_{0}\Vert
0,y-Ty,Ty-y\Vert \text{.}  \label{eqn7}
\end{equation}%
For each $i\in \overline{I_{1}}$, $x_{i}\in F_{m-1}=conv(F_{m-2}\cup
TF_{m-2})$, there is a finite set $J_{i}$, such that
\begin{equation}
x_{i}=\sum\limits_{j\in J_{i}}\beta _{i}^{j}x_{i}^{j}\text{,}  \label{eqn8}
\end{equation}%
with $x_{i}^{j}\in F_{m-2}\cup TF_{m-2}$, $\beta _{i}^{j}\geq 0$ and $%
\sum\limits_{j\in J_{i}}\beta _{i}^{j}=1$ for any $j\in J_{i}$. Let $%
J_{i}=J_{i}^{1}\cup J_{i}^{2}$, with $J_{i}^{1}\cap J_{i}^{2}=\emptyset $
such that%
\begin{equation}
x_{i}=\sum\limits_{j\in J_{i}^{1}}\beta _{i}^{j}x_{i}^{j}+\sum\limits_{j\in
J_{i}^{2}}\beta _{i}^{j}Tx_{i}^{j}\text{.}  \label{eqn9}
\end{equation}%
For each $i\in \overline{I_{1}}$ we have%
\begin{equation}
\Vert 0,x_{i}-Ty,Ty-x_{i}\Vert \leq \sum\limits_{j\in J_{i}^{1}}\beta
_{i}^{j}\Vert 0,x_{i}^{j}-Ty,Ty-x_{i}^{j}\Vert +\sum\limits_{j\in
J_{i}^{2}}\beta _{i}^{j}\Vert 0,Tx_{i}^{j}-Ty,Ty-Tx_{i}^{j}\Vert \text{.}
\label{eqn10}
\end{equation}%
Applying the condition $(\mathbf{NS25})$ to $\Vert
0,Tx_{i}^{j}-Ty,Ty-Tx_{i}^{j}\Vert $, we have%
\begin{equation}
\begin{array}{r}
\Vert 0,Tx_{i}^{j}-Ty,Ty-Tx_{i}^{j}\Vert <\max \{\Vert
0,x_{i}^{j}-y,y-x_{i}^{j}\Vert ,\Vert
0,Tx_{i}^{j}-x_{i}^{j},x_{i}^{j}-Tx_{i}^{j}\Vert , \\
\Vert 0,y-Ty,Ty-y\Vert ,\Vert 0,Tx_{i}^{j}-y,y-Tx_{i}^{j}\Vert ,\Vert
0,x_{i}^{j}-Ty,Ty-x_{i}^{j}\Vert \}\text{.}%
\end{array}
\label{eqn11}
\end{equation}%
Since $x_{i}^{j}\in F_{m-2}$ and $y\in F_{m}$, we have%
\begin{equation*}
\begin{array}{l}
\Vert 0,x_{i}^{j}-y,y-x_{i}^{j}\Vert \leq r^{s}(F), \\
\Vert 0,Tx_{i}^{j}-x_{i}^{j},x_{i}^{j}-Tx_{i}^{j}\Vert \leq r^{s}(F), \\
\Vert 0,Tx_{i}^{j}-y,y-Tx_{i}^{j}\Vert \leq r^{s}(F).%
\end{array}%
\end{equation*}%
By (\ref{eqn11}), we obtain%
\begin{equation*}
\Vert 0,Tx_{i}^{j}-Ty,Ty-Tx_{i}^{j}\Vert <\max \{r^{s}(F),\Vert
0,y-Ty,Ty-y\Vert ,\Vert 0,x_{i}^{j}-Ty,Ty-x_{i}^{j}\Vert \}\text{.}
\end{equation*}%
Let $J_{i}^{2}=J_{i}^{2_{1}}\cup J_{i}^{2_{2}}\cup J_{i}^{2_{3}}$ with $%
J_{i}^{2_{k}}\cap J_{i}^{2_{p}}=\emptyset $ such that%
\begin{equation*}
\begin{array}{l}
J_{i}^{2_{1}}=\{j\in J_{i}^{2}:\Vert 0,Tx_{i}^{j}-Ty,Ty-Tx_{i}^{j}\Vert
<r^{s}(F)\}, \\
J_{i}^{2_{2}}=\{j\in J_{i}^{2}:\Vert 0,Tx_{i}^{j}-Ty,Ty-Tx_{i}^{j}\Vert
<\Vert 0,y-Ty,Ty-y\Vert \}, \\
J_{i}^{2_{3}}=\{j\in J_{i}^{2}:\Vert 0,Tx_{i}^{j}-Ty,Ty-Tx_{i}^{j}\Vert
<\Vert 0,x_{i}^{j}-Ty,Ty-x_{i}^{j}\Vert \}.%
\end{array}%
\end{equation*}%
Using (\ref{eqn10}), we obtain%
\begin{equation}
\begin{array}{c}
\Vert 0,x_{i}-Ty,Ty-x_{i}\Vert \leq \sum\limits_{j\in J_{i}^{1}\cup
J_{i}^{2_{3}}}\beta _{i}^{j}\Vert 0,x_{i}^{j}-Ty,Ty-x_{i}^{j}\Vert
+\sum\limits_{i\in J_{i}^{2_{1}}}\beta _{i}^{j}r^{s}(F) \\
+\sum\limits_{i\in J_{i}^{2_{2}}}\beta _{i}^{j}\Vert 0,y-Ty,Ty-y\Vert .%
\end{array}
\label{eqn12}
\end{equation}%
Let us denote by%
\begin{equation*}
J_{1}^{i}=J_{i}^{1}\cup J_{i}^{2_{3}}\text{, }J_{2}^{i}=J_{i}^{2_{2}}\text{
and }J_{3}^{i}=J_{i}^{2_{1}}\text{.}
\end{equation*}%
Then using (\ref{eqn12}), we have%
\begin{equation}
\begin{array}{c}
\Vert 0,x_{i}-Ty,Ty-x_{i}\Vert \leq \sum\limits_{j\in J_{1}^{i}}\beta
_{i}^{j}\Vert 0,x_{i}^{j}-Ty,Ty-x_{i}^{j}\Vert +\sum\limits_{i\in
J_{3}^{i}}\beta _{i}^{j}r^{s}(F) \\
+\sum\limits_{i\in J_{2}^{i}}\beta _{i}^{j}\Vert 0,y-Ty,Ty-y\Vert .%
\end{array}
\label{eqn13}
\end{equation}%
If $A_{i}=\sum\limits_{i\in J_{3}^{i}}\beta _{i}^{j}$ and $%
B_{i}=\sum\limits_{i\in J_{2}^{i}}\beta _{i}^{j}$ with $\sum\limits_{j\in
J_{1}^{i}}\beta _{i}^{j}+A_{i}+B_{i}=1$. Using (\ref{eqn13}), we obtain%
\begin{equation}
\Vert 0,x_{i}-Ty,Ty-x_{i}\Vert \leq \sum\limits_{j\in J_{1}^{i}}\beta
_{i}^{j}\Vert 0,x_{i}^{j}-Ty,Ty-x_{i}^{j}\Vert +A_{i}r^{s}(F)+B_{i}\Vert
0,y-Ty,Ty-y\Vert .  \label{eqn14}
\end{equation}%
Using (\ref{eqn7}) and $\Vert 0,x_{i}-Ty,Ty-x_{i}\Vert $ by (\ref{eqn14}),
we obtain%
\begin{equation*}
\begin{array}{c}
\Vert 0,x-Ty,Ty-x\Vert \leq \sum\limits_{i\in \overline{I_{1}}}\alpha
_{i}\sum\limits_{j\in J_{1}^{i}}\beta _{i}^{j}\Vert
0,x_{i}^{j}-Ty,Ty-x_{i}^{j}\Vert +\left[ \sum\limits_{i\in \overline{I_{1}}%
}\alpha _{i}A_{i}+A_{0}\right] r^{s}(F) \\
+\left[ \sum\limits_{i\in \overline{I_{1}}}\alpha _{i}B_{i}+B_{0}\right]
\Vert 0,y-Ty,Ty-y\Vert .%
\end{array}%
\end{equation*}%
Let $A_{1}=\sum\limits_{i\in \overline{I_{1}}}\alpha _{i}A_{i}$ and $%
B_{1}=\sum\limits_{i\in \overline{I_{1}}}\alpha _{i}B_{i}$. Then we have%
\begin{equation}
\begin{array}{c}
\Vert 0,x-Ty,Ty-x\Vert \leq \sum\limits_{i\in \overline{I_{1}}}\alpha
_{i}\sum\limits_{j\in J_{1}^{i}}\beta _{i}^{j}\Vert
0,x_{i}^{j}-Ty,Ty-x_{i}^{j}\Vert +(A_{1}+A_{0})r^{s}(F) \\
+(B_{1}+B_{0})\Vert 0,y-Ty,Ty-y\Vert .%
\end{array}
\label{eqn15}
\end{equation}%
We note that%
\begin{equation*}
\sum\limits_{i\in \overline{I_{1}}}\alpha _{i}\sum\limits_{j\in
J_{1}^{i}}\beta _{i}^{j}+\sum\limits_{i\in \overline{I_{1}}}\alpha
_{i}A_{i}+A_{0}+\sum\limits_{i\in \overline{I_{1}}}\alpha _{i}B_{i}+B_{0}=1.
\end{equation*}%
Let us take $K=\underset{i\in \overline{I_{1}}}{\bigcup }\left( \underset{%
j\in J_{1}^{i}}{\bigcup }j\right) $ and denote the scalars by $\xi _{k}$. To
each $k$ relative to the pair $(i,j)$, $x_{i}^{j}$ will be denoted by $x_{k}$%
.

Using (\ref{eqn15}), we obtain%
\begin{equation*}
\begin{array}{c}
\Vert 0,x-Ty,Ty-x\Vert \leq \sum\limits_{k\in K}\xi _{k}\Vert
0,x_{k}-Ty,Ty-x_{k}\Vert +(A_{1}+A_{0})r^{s}(F) \\
+(B_{1}+B_{0})\Vert 0,y-Ty,Ty-y\Vert ,%
\end{array}%
\end{equation*}%
where $\sum\limits_{k\in K}\xi _{k}+A_{1}+A_{0}+B_{1}+B_{0}=1$ and $x_{k}\in
F_{m-2}$.

Repeating this process which is done for $x_{k}$, we get%
\begin{equation}
\begin{array}{c}
\Vert 0,x-Ty,Ty-x\Vert \leq \sum\limits_{p\in P}\gamma _{p}\Vert
0,x_{p}-Ty,Ty-x_{p}\Vert +\sum\limits_{k=0}^{m-1}A_{k}r^{s}(F) \\
+\sum\limits_{k=0}^{m-1}B_{k}\Vert 0,y-Ty,Ty-y\Vert ,%
\end{array}
\label{eqn16}
\end{equation}%
where $\sum\limits_{p\in P}\gamma
_{p}+\sum\limits_{i=0}^{m-1}(A_{i}+B_{i})=1 $ and $x_{p}\in F_{1}=C^{s}(F)$.

Hence $\Vert 0,x_{p}-Ty,Ty-x_{p}\Vert \leq r^{s}(F)$ and using (\ref{eqn16}%
), we obtain%
\begin{equation}
\Vert 0,x-Ty,Ty-x\Vert \leq \sum\limits_{k=0}^{m-1}B_{k}\Vert
0,y-Ty,Ty-y\Vert +\left( \sum\limits_{p\in P}\gamma
_{p}+\sum\limits_{k=0}^{m-1}A_{k}\right) r^{s}(F).  \label{eqn17}
\end{equation}%
Let us turn to $\Vert 0,y-Ty,Ty-y\Vert $. Since $y\in conv(F_{m-1}\cup
TF_{m-1})$, we have $y=\sum\limits_{i\in I}\alpha _{i}y_{i}$ with $%
\sum\limits_{i\in I}\alpha _{i}=1$, $y_{i}\in F_{m-1}\cup TF_{m-1}$ and $%
\alpha _{i}\geq 0$ for all $i\in I$. Let $I=I_{1}\cup I_{2}$ such that $%
I_{1}\cap I_{2}=\emptyset $. If $i\in I_{1}$ then $y_{i}\in F_{m-1}$ and if $%
i\in I_{2}$ then $y_{i}\in TF_{m-1}$. Let $y_{i}=Ty_{i}$. Then we can write%
\begin{equation*}
y=\sum\limits_{i\in I_{1}}\alpha _{i}y_{i}+\sum\limits_{i\in I_{2}}\alpha
_{i}Ty_{i},
\end{equation*}%
with $y_{i}\in F_{m-1}$.

Substituting in $\Vert 0,y-Ty,Ty-y\Vert $ we get%
\begin{equation}
\Vert 0,y-Ty,Ty-y\Vert \leq \sum\limits_{i\in I_{1}}\alpha _{i}\Vert
0,y_{i}-Ty,Ty-y_{i}\Vert +\sum\limits_{i\in I_{2}}\alpha _{i}\Vert
0,Ty_{i}-Ty,Ty-Ty_{i}\Vert .  \label{eqn18}
\end{equation}%
Using the condition $(\mathbf{NS}25)$, we obtain%
\begin{equation}
\begin{array}{r}
\Vert 0,Ty_{i}-Ty,Ty-Ty_{i}\Vert <\max \{\Vert 0,y_{i}-y,y-y_{i}\Vert ,\Vert
0,y_{i}-Ty_{i},Ty_{i}-y_{i}\Vert , \\
\Vert 0,y-Ty,Ty-y\Vert ,\Vert 0,y_{i}-Ty,Ty-y_{i}\Vert ,\Vert
0,Ty_{i}-y,y-Ty_{i}\Vert \}\text{.}%
\end{array}
\label{eqn19}
\end{equation}%
Since $y_{i}\in F_{m-1}$, $y\in F_{m}$, we have%
\begin{equation*}
\begin{array}{l}
\Vert 0,y_{i}-y,y-y_{i}\Vert \leq r^{s}(F), \\
\Vert 0,y_{i}-Ty_{i},Ty_{i}-y_{i}\Vert \leq r^{s}(F), \\
\Vert 0,Ty_{i}-y,y-Ty_{i}\Vert \leq r^{s}(F).%
\end{array}%
\end{equation*}%
Using (\ref{eqn19}), we obtain%
\begin{equation*}
\Vert 0,Ty_{i}-Ty,Ty-Ty_{i}\Vert <\max \{r^{s}(F),\Vert 0,y-Ty,Ty-y\Vert
,\Vert 0,y_{i}-Ty,Ty-y_{i}\Vert \}\text{.}
\end{equation*}%
Redefining the index set $I_{2}=I_{2}^{1}\cup I_{2}^{2}\cup I_{2}^{3}$ with%
\begin{equation*}
\begin{array}{l}
I_{2}^{1}=\{i\in I_{2}:\Vert 0,Ty_{i}-Ty,Ty-Ty_{i}\Vert <r^{s}(F)\}, \\
I_{2}^{2}=\{i\in I_{2}:\Vert 0,Ty_{i}-Ty,Ty-Ty_{i}\Vert <\Vert
0,y-Ty,Ty-y\Vert \}, \\
I_{2}^{3}=\{i\in I_{2}:\Vert 0,Ty_{i}-Ty,Ty-Ty_{i}\Vert <\Vert
0,y_{i}-Ty,Ty-y_{i}\Vert \}.%
\end{array}%
\end{equation*}%
Now using (\ref{eqn18}), we get
\begin{equation}
\begin{array}{c}
\Vert 0,y-Ty,Ty-y\Vert \leq \sum\limits_{i\in I_{1}\cup I_{2}^{3}}\alpha
_{i}\Vert 0,y_{i}-Ty,Ty-y_{i}\Vert +\sum\limits_{i\in I_{2}^{1}}\alpha
_{i}r^{s}(F) \\
+\sum\limits_{i\in I_{2}^{2}}\alpha _{i}\Vert 0,y-Ty,Ty-y\Vert .%
\end{array}
\label{eqn20}
\end{equation}%
We note that if $\sum\limits_{i\in I_{2}^{2}}\alpha _{i}=1$ then%
\begin{equation*}
\Vert 0,y-Ty,Ty-y\Vert \leq \Vert 0,Ty_{i}-Ty,Ty-Ty_{i}\Vert <\Vert
0,y-Ty,Ty-y\Vert ,
\end{equation*}%
which is a contradiction.

Then $\sum\limits_{i\in I_{2}^{2}}\alpha _{i}<1$ and using (\ref{eqn20}), we
obtain%
\begin{equation}
\begin{array}{c}
\Vert 0,y-Ty,Ty-y\Vert \leq \sum\limits_{i\in I_{1}\cup I_{2}^{3}}\dfrac{%
\alpha _{i}}{1-\sum\limits_{i\in I_{2}^{2}}\alpha _{i}}\Vert
0,y_{i}-Ty,Ty-y_{i}\Vert \\
+\sum\limits_{i\in I_{2}^{1}}\dfrac{\alpha _{i}}{1-\sum\limits_{i\in
I_{2}^{2}}\alpha _{i}}r^{s}(F),%
\end{array}
\label{eqn21}
\end{equation}%
with $\sum\limits_{i\in I_{1}\cup I_{2}^{3}}\dfrac{\alpha _{i}}{%
1-\sum\limits_{i\in I_{2}^{2}}\alpha _{i}}+\sum\limits_{i\in I_{2}^{1}}%
\dfrac{\alpha _{i}}{1-\sum\limits_{i\in I_{2}^{2}}\alpha _{i}}=1$.

Let $I_{1}=I_{1}\cup I_{2}^{3}$, $I_{2}=I_{2}^{1}$ and $\beta _{i}=\dfrac{%
\alpha _{i}}{1-\sum\limits_{i\in I_{2}^{2}}\alpha _{i}}$. Using (\ref{eqn21}%
), we obtain%
\begin{equation}
\Vert 0,y-Ty,Ty-y\Vert \leq \sum\limits_{i\in I_{1}}\beta _{i}\Vert
0,y_{i}-Ty,Ty-y_{i}\Vert +\sum\limits_{i\in I_{2}}\beta _{i}r^{s}(F).
\label{eqn22}
\end{equation}%
If $A_{0}=\sum\limits_{i\in I_{2}}\beta _{i}$ then using (\ref{eqn22}), we
have%
\begin{equation}
\Vert 0,y-Ty,Ty-y\Vert \leq \sum\limits_{i\in I_{1}}\beta _{i}\Vert
0,y_{i}-Ty,Ty-y_{i}\Vert +A_{0}r^{s}(F),  \label{eqn23}
\end{equation}%
with $\sum\limits_{i\in I_{1}}\beta _{i}+A_{0}=1$ and $y_{i}\in
F_{m-1}=conv(F_{m-2}\cup TF_{m-2})$.

For each $i\in I_{1}$,%
\begin{equation*}
y_{i}=\sum\limits_{j\in J_{i}^{1}}\gamma _{i}^{j}y_{i}^{j}+\sum\limits_{j\in
J_{i}^{2}}\gamma _{i}^{j}Ty_{i}^{j},
\end{equation*}%
with $y_{i}^{j}\in F_{m-2}$ and $\sum\limits_{j\in J_{i}^{1}\cup
J_{i}^{2}}\gamma _{i}^{j}=1$. So we obtain%
\begin{equation}
\Vert 0,y_{i}-Ty,Ty-y_{i}\Vert \leq \sum\limits_{j\in J_{i}^{1}}\gamma
_{i}^{j}\Vert 0,y_{i}^{j}-Ty,Ty-y_{i}^{j}\Vert +\sum\limits_{j\in
J_{i}^{2}}\gamma _{i}^{j}\Vert 0,Ty_{i}^{j}-Ty,Ty-Ty_{i}^{j}\Vert .
\label{eqn24}
\end{equation}%
Using the condition $(\mathbf{NS25}),$ we have%
\begin{equation*}
\begin{array}{r}
\Vert 0,Ty_{i}^{j}-Ty,Ty-Ty_{i}^{j}\Vert <\max \{\Vert
0,y_{i}^{j}-y,y-y_{i}^{j}\Vert ,\Vert
0,Ty_{i}^{j}-y_{i}^{j},y_{i}^{j}-Ty_{i}^{j}\Vert , \\
\Vert 0,y-Ty,Ty-y\Vert ,\Vert 0,Ty_{i}^{j}-y,y-Ty_{i}^{j}\Vert ,\Vert
0,y_{i}^{j}-Ty,Ty-y_{i}^{j}\Vert \}\text{.}%
\end{array}%
\end{equation*}%
Since $y_{i}^{j}\in F_{m-2}$, $Ty_{i}^{j}\in F_{m-1}$ and $y\in F_{m}$, we
can write%
\begin{equation*}
\begin{array}{l}
\Vert 0,y_{i}^{j}-y,y-y_{i}^{j}\Vert \leq r^{s}(F), \\
\Vert 0,Ty_{i}^{j}-y_{i}^{j},y_{i}^{j}-Ty_{i}^{j}\Vert \leq r^{s}(F), \\
\Vert 0,Ty_{i}^{j}-y,y-Ty_{i}^{j}\Vert \leq r^{s}(F)%
\end{array}%
\end{equation*}%
and%
\begin{equation*}
\Vert 0,Ty_{i}^{j}-Ty,Ty-Ty_{i}^{j}\Vert <\max \{r^{s}(F),\Vert
0,y-Ty,Ty-y\Vert ,\Vert 0,y_{i}^{j}-Ty,Ty-y_{i}^{j}\Vert \}\text{.}
\end{equation*}%
Let $J_{i}^{2}$ be the union of the disjoint sets $J_{i}^{2}=J_{i}^{2_{1}}%
\cup J_{i}^{2_{2}}\cup J_{i}^{2_{3}}$ such that%
\begin{equation*}
\begin{array}{l}
J_{i}^{2_{1}}=\{j\in J_{i}^{2}:\Vert 0,Ty_{i}^{j}-Ty,Ty-Ty_{i}^{j}\Vert
<\Vert 0,y_{i}^{j}-Ty,Ty-y_{i}^{j}\Vert \}, \\
J_{i}^{2_{2}}=\{j\in J_{i}^{2}:\Vert 0,Ty_{i}^{j}-Ty,Ty-Ty_{i}^{j}\Vert
<r^{s}(F)\}, \\
J_{i}^{2_{3}}=\{j\in J_{i}^{2}:\Vert 0,Ty_{i}^{j}-Ty,Ty-Ty_{i}^{j}\Vert
<\Vert 0,y-Ty,Ty-y\Vert \}.%
\end{array}%
\end{equation*}%
Using (\ref{eqn24}), we obtain%
\begin{equation}
\begin{array}{c}
\Vert 0,y_{i}-Ty,Ty-y_{i}\Vert \leq \sum\limits_{j\in J_{i}^{1}\cup
J_{i}^{2_{1}}}\gamma _{i}^{j}\Vert 0,y_{i}^{j}-Ty,Ty-y_{i}^{j}\Vert
+\sum\limits_{j\in J_{i}^{2_{2}}}\gamma _{i}^{j}r^{s}(F) \\
+\sum\limits_{j\in J_{i}^{2_{3}}}\gamma _{i}^{j}\Vert 0,y-Ty,Ty-y\Vert .%
\end{array}
\label{eqn25}
\end{equation}%
Now redefine the index sets $J_{i}^{1}=J_{i}^{1}\cup J_{i}^{2_{1}}$, $%
J_{i}^{2}=J_{i}^{2_{2}},$ $J_{i}^{3}=J_{i}^{2_{3}}$ and using (\ref{eqn25})
we can write%
\begin{equation*}
\begin{array}{c}
\Vert 0,y_{i}-Ty,Ty-y_{i}\Vert \leq \sum\limits_{j\in J_{i}^{1}}\gamma
_{i}^{j}\Vert 0,y_{i}^{j}-Ty,Ty-y_{i}^{j}\Vert +\sum\limits_{j\in
J_{i}^{2}}\gamma _{i}^{j}r^{s}(F) \\
+\sum\limits_{j\in J_{i}^{3}}\gamma _{i}^{j}\Vert 0,y-Ty,Ty-y\Vert ,%
\end{array}%
\end{equation*}%
with $\sum\limits_{j\in J_{i}^{1}}\gamma _{i}^{j}+\sum\limits_{j\in
J_{i}^{2}}\gamma _{i}^{j}+\sum\limits_{j\in J_{i}^{3}}\gamma _{i}^{j}=1$.

Using the (\ref{eqn23}), we obtain%
\begin{equation}
\begin{array}{c}
\Vert 0,y-Ty,Ty-y\Vert \leq \sum\limits_{i\in I_{1}}\beta
_{i}\sum\limits_{j\in J_{i}^{1}}\gamma _{i}^{j}\Vert
0,y_{i}^{j}-Ty,Ty-y_{i}^{j}\Vert +\sum\limits_{i\in I_{1}}\beta
_{i}\sum\limits_{j\in J_{i}^{3}}\gamma _{i}^{j}\Vert 0,y-Ty,Ty-y\Vert \\
+\left[ \sum\limits_{i\in I_{1}}\beta _{i}\sum\limits_{j\in J_{i}^{2}}\gamma
_{i}^{j}+A_{0}\right] r^{s}(F).%
\end{array}
\label{eqn26}
\end{equation}%
If $\sum\limits_{i\in I_{1}}\beta _{i}\sum\limits_{j\in J_{i}^{3}}\gamma
_{i}^{j}=1$ we have%
\begin{equation*}
\Vert 0,y-Ty,Ty-y\Vert <\Vert 0,y-Ty,Ty-y\Vert ,
\end{equation*}%
which is a contradiction.

Hence $\sum\limits_{i\in I_{1}}\beta _{i}\sum\limits_{j\in J_{i}^{3}}\gamma
_{i}^{j}<1$ and using (\ref{eqn26}), we obtain%
\begin{equation}
\begin{array}{c}
\Vert 0,y-Ty,Ty-y\Vert \leq \dfrac{\sum\limits_{i\in I_{1}}\beta
_{i}\sum\limits_{j\in J_{i}^{1}}\gamma _{i}^{j}}{1-\sum\limits_{i\in
I_{1}}\beta _{i}\sum\limits_{j\in J_{i}^{3}}\gamma _{i}^{j}}\Vert
0,y_{i}^{j}-Ty,Ty-y_{i}^{j}\Vert \\
+\dfrac{\sum\limits_{i\in I_{1}}\beta _{i}\sum\limits_{j\in J_{i}^{2}}\gamma
_{i}^{j}+A_{0}}{1-\sum\limits_{i\in I_{1}}\beta _{i}\sum\limits_{j\in
J_{i}^{3}}\gamma _{i}^{j}}r^{s}(F),%
\end{array}
\label{eqn27}
\end{equation}%
with $\dfrac{\sum\limits_{i\in I_{1}}\beta _{i}\sum\limits_{j\in
J_{i}^{1}}\gamma _{i}^{j}}{1-\sum\limits_{i\in I_{1}}\beta
_{i}\sum\limits_{j\in J_{i}^{3}}\gamma _{i}^{j}}+\dfrac{\sum\limits_{i\in
I_{1}}\beta _{i}\sum\limits_{j\in J_{i}^{2}}\gamma _{i}^{j}+A_{0}}{%
1-\sum\limits_{i\in I_{1}}\beta _{i}\sum\limits_{j\in J_{i}^{3}}\gamma
_{i}^{j}}=1$.

Let $A_{1}=\dfrac{\sum\limits_{i\in I_{1}}\beta _{i}\sum\limits_{j\in
J_{i}^{2}}\gamma _{i}^{j}+A_{0}}{1-\sum\limits_{i\in I_{1}}\beta
_{i}\sum\limits_{j\in J_{i}^{3}}\gamma _{i}^{j}}$ and denote the index set
by $K=\underset{i\in I_{1}}{\bigcup }\left( \underset{j\in J_{i}^{1}}{%
\bigcup }j\right) $, write $\zeta _{k}$ for $k\in K$ relative to $(i,j)$,
that is%
\begin{equation*}
\zeta _{k}=\dfrac{\sum\limits_{i\in I_{1}}\beta _{i}\sum\limits_{j\in
J_{i}^{2}}\gamma _{i}^{j}}{1-\sum\limits_{i\in I_{1}}\beta
_{i}\sum\limits_{j\in J_{i}^{3}}\gamma _{i}^{j}}\text{.}
\end{equation*}%
Also we write $y_{k}$ for $y_{i}^{j}$. Then using (\ref{eqn27}), we obtain%
\begin{equation*}
\Vert 0,y-Ty,Ty-y\Vert \leq \sum\limits_{k\in K}\zeta _{k}\Vert
0,y_{k}-Ty,Ty-y_{k}\Vert +(A_{1}+A_{0})r^{s}(F),
\end{equation*}%
with $\sum\limits_{k\in K}\zeta _{k}+A_{1}+A_{0}=1$ and $y_{k}\in F_{m-2}$.

Repeating this process we get
\begin{equation*}
\Vert 0,y-Ty,Ty-y\Vert \leq \sum\limits_{p\in P}\lambda _{p}\Vert
0,y_{p}-Ty,Ty-y_{p}\Vert +\sum\limits_{k=0}^{m-1}A_{k}r^{s}(F),
\end{equation*}%
where $y_{p}\in F_{1}$ and $\sum\limits_{p\in P}\lambda
_{p}+\sum\limits_{k=0}^{m-1}A_{k}=1$.

Then $\Vert 0,y_{p}-Ty,Ty-y_{p}\Vert \leq r^{s}(F)$ and%
\begin{equation*}
\Vert 0,y-Ty,Ty-y\Vert \leq \left( \sum\limits_{p\in P}\lambda
_{p}+\sum\limits_{k=0}^{m-1}A_{k}\right) r^{s}(F)=r^{s}(F).
\end{equation*}%
Using (\ref{eqn17}), we get%
\begin{equation*}
\Vert 0,x-Ty,Ty-x\Vert \leq \left(
\sum\limits_{k=0}^{m-1}B_{k}+\sum\limits_{p\in P}\gamma
_{p}+\sum\limits_{k=0}^{m-1}A_{k}\right) r^{s}(F),
\end{equation*}%
with $\sum\limits_{k=0}^{m-1}B_{k}+\sum\limits_{p\in P}\gamma
_{p}+\sum\limits_{k=0}^{m-1}A_{k}=1$.

Consequently, we obtain $\Vert 0,x-Ty,Ty-x\Vert \leq r^{s}(F)$ and so%
\begin{equation*}
\delta _{m+1}^{s}-\varepsilon <\Vert 0,x-Ty,Ty-x\Vert \leq r^{s}(F).
\end{equation*}

\textbf{Case 3}. For $x,y\in F_{m}$, we have%
\begin{equation}
\begin{array}{c}
\delta _{m+1}^{s}-\varepsilon <\Vert 0,Tx-Ty,Ty-Tx\Vert <\max \{\Vert
0,x-y,y-x\Vert ,\Vert 0,Tx-x,x-Tx\Vert , \\
\Vert 0,y-Ty,Ty-y\Vert ,\Vert 0,x-Ty,Ty-x\Vert ,\Vert 0,Tx-y,y-Tx\Vert \}%
\end{array}
\label{eqn28}
\end{equation}%
and repeating what has been done in \textbf{Case 2}, we get%
\begin{equation*}
\delta _{m+1}^{s}-\varepsilon <\Vert 0,Tx-Ty,Ty-Tx\Vert \leq r^{s}(F).
\end{equation*}%
In all three cases we have $\delta _{m+1}^{s}-\varepsilon <r^{s}(F)$. If $%
\varepsilon $ tends to $0$ we get $\delta _{m+1}^{s}\leq r^{s}(F)$.

Let $F^{\infty }=\bigcup\limits_{n\in
\mathbb{N}
}F_{n}$. Then $F^{\infty }$ is nonempty because $C^{s}(F)\neq \emptyset $.
Since $F_{k}\subseteq F_{k+1}$, we obtain%
\begin{equation*}
\delta ^{s}(F^{\infty })=\lim_{k\rightarrow \infty }\delta ^{s}(F_{k})\leq
r^{s}(F).
\end{equation*}%
As $F_{k}\subset F$, $F^{\infty }\subseteq F$ and so $\delta ^{s}(F^{\infty
})\leq r^{s}(F)$.

Using the $S$-normal structure of $F$ we have $r^{s}(F)<\delta ^{s}(F)$ and $%
\delta ^{s}(F^{\infty })<\delta ^{s}(F)$. So $F^{\infty }$ must be a proper
subset of $F$. We obtain that $F^{\infty }$ is convex and $TF^{\infty
}\subseteq F^{\infty }$.

Let $M=\overline{convF^{\infty }}=\overline{F^{\infty }}$, its diameter is
the same as $F^{\infty }$. So we have%
\begin{equation*}
\delta ^{s}(M)\leq r^{s}(F)<\delta ^{s}(F)
\end{equation*}%
and $M$ is closed, nonempty and convex proper subset of $F$. Since $T$ is
continuous then $M$ is $T$-invariant and%
\begin{equation*}
TM=TF^{\infty }\subseteq \overline{TF^{\infty }}\subseteq \overline{%
F^{\infty }}=M.
\end{equation*}%
So $M\in \mathcal{A}$ and $M\subsetneqq F$ contradicting the minimality of $%
F $. Hence, it should be $\delta ^{s}(F)=0$. Consequently, $F$ has a unique
fixed point under $T$.
\end{proof}

\section{Some Comparisons on $S$-Normed Spaces}

\label{sec:4} In \cite{nihal}, the present authors defined Rhoades'
condition $(\mathbf{S25})$ using the notion of an $S$-metric. Also, they
investigated relationships between the conditions $(\mathbf{S25})$ and $(%
\mathbf{R25})$ in \cite{nihal2}.

In this section, we determine the relationships between the conditions $(%
\mathbf{S25})$ (resp. $(\mathbf{NR25})$) and $(\mathbf{NS25})$.

At first, we recall the Rhoades' condition on normed spaces as follows \cite%
{paula}:

Let $(X,\Vert .\Vert )$ be a Banach space and $T$ be a self-mapping of $X$.%
\begin{equation*}
(\mathbf{NR25})\text{ \ \ \ \ }\Vert Tx-Ty\Vert <\max \{\Vert x-y\Vert
,\Vert x-Tx\Vert ,\Vert y-Ty\Vert ,\Vert x-Ty\Vert ,\Vert y-Tx\Vert \},
\end{equation*}%
for each $x,y\in X$, $x\neq y$.

Now we give the relationship between $(\mathbf{S25})$ and $(\mathbf{NS25})$
in the following proposition.

\begin{proposition}
\label{prop7} Let $(X,\Vert .,.,.\Vert )$ be an $S$-Banach space, $%
(X,S_{\Vert .\Vert })$ be the $S$-metric space obtained by the $S$-metric
generated by $\Vert .,.,.\Vert $ and $T$ be a self-mapping of $X$. If $T$
satisfies the condition $(\mathbf{NS25})$ then $T$ satisfies the condition $(%
\mathbf{S25})$.
\end{proposition}

\begin{proof}
Assume that $T$ satisfies the condition $(\mathbf{NS25})$. Using the
condition $(\mathbf{NS25})$, we have%
\begin{equation*}
\begin{array}{l}
S_{\Vert .\Vert }(Tx,Tx,Ty)=\Vert Tx-Tx,Tx-Ty,Ty-Tx\Vert =\Vert
0,Tx-Ty,Ty-Tx\Vert \\
<\max \{\Vert 0,x-y,y-x\Vert ,\Vert 0,Tx-x,x-Tx\Vert ,\Vert 0,Ty-y,y-Ty\Vert
, \\
\Vert 0,Ty-x,x-Ty\Vert ,\Vert 0,Tx-y,y-Tx\Vert \} \\
=\max \{S_{\Vert .\Vert }(x,x,y),S_{\Vert .\Vert }(Tx,Tx,x),S_{\Vert .\Vert
}(Ty,Ty,y),S_{\Vert .\Vert }(Ty,Ty,x),S_{\Vert .\Vert }(Tx,Tx,y)\}%
\end{array}%
\end{equation*}%
and so the condition $(\mathbf{S25})$ is satisfied by $T$ on $(X,S_{\Vert
.\Vert })$.
\end{proof}

Now, we give the relationship between the conditions $(\mathbf{NR25})$ and $(%
\mathbf{NS25})$ in the following proposition.

\begin{proposition}
\label{prop8} Let $(X,\Vert .\Vert )$ be a Banach space, $(X,\Vert
.,.,.\Vert )$ be an $S$-normed space obtained by the $S$-norm generated by $%
\Vert .\Vert $ and $T$ be a self-mapping of $X$. If $T$ satisfies the
condition $(\mathbf{NR25})$ then $T$ satisfies the condition $(\mathbf{NS25}%
) $.
\end{proposition}

\begin{proof}
Let $T$ satisfies the condition $(\mathbf{NR25})$. Using the conditions $(%
\mathbf{NR25})$ and $(\mathbf{N3})$, we have%
\begin{equation*}
\begin{array}{l}
\Vert 0,Tx-Ty,Ty-Tx\Vert =\Vert 0\Vert +\Vert Tx-Ty\Vert +\Vert Ty-Tx\Vert
=2\Vert Tx-Ty\Vert \\
<2\max \{\Vert x-y\Vert ,\Vert x-Tx\Vert ,\Vert y-Ty\Vert ,\Vert x-Ty\Vert
,\Vert y-Tx\Vert \} \\
=\max \{2\Vert x-y\Vert ,2\Vert x-Tx\Vert ,2\Vert y-Ty\Vert ,2\Vert
x-Ty\Vert ,2\Vert y-Tx\Vert \} \\
=\max \{\Vert x-y\Vert +\Vert y-x\Vert ,\Vert x-Tx\Vert +\Vert Tx-x\Vert
,\Vert y-Ty\Vert +\Vert Ty-y\Vert , \\
\Vert x-Ty\Vert +\Vert Ty-x\Vert ,\Vert y-Tx\Vert +\Vert Tx-y\Vert \} \\
=\max \{\Vert 0,x-y,y-x\Vert ,\Vert 0,Tx-x,x-Tx\Vert ,\Vert 0,Ty-y,y-Ty\Vert
, \\
\Vert 0,Ty-x,x-Ty\Vert ,\Vert 0,Tx-y,y-Tx\Vert \}%
\end{array}%
\end{equation*}%
and so the condition $(\mathbf{NS25})$ is satisfied.
\end{proof}

Finally, we give the relationship between Theorem \ref{thm1} and the
following theorem.

\begin{theorem}
\cite{paula} \label{thm2} Let $X$ be a reflexive Banach space and $A$ be a
nonempty, closed, bounded and convex subset of $X$, having normal structure.
If $T:A\rightarrow A$ is a continuous self-mapping satisfying the condition $%
(\mathbf{NR25})$ then $T$ has a unique fixed point in $A$.
\end{theorem}

Theorem \ref{thm1} and Theorem \ref{thm2} are coincide when $X$ is an $S$%
-Banach space obtained by the $S$-norm generated by $\Vert .\Vert $.
Clearly, Theorem \ref{thm1} is a generalization of Theorem \ref{thm2} as we
have seen in Section \ref{sec:2} that there are $S$-norms which are not
generated by any norm.


\begin{thebibliography}{99}
\bibitem{chang} S. S. Chang, \textit{On Rhoades' Open Questions and Some
Fixed Point Theorems for a Class of Mappings, }Proc. Amer. Math. Soc. 97 (2)
(1986) 343-346.

\bibitem{sen} S. S. Chang and Q. C. Zhong, {\textit{On Rhoades' Open
Questions},} Proc. Amer. Math. Soc. {109 (1)} (1990) 269-274.

\bibitem{O Ege} \"{O}. Ege, C. Park and A. H. Ansari, A different approach
to complex valued $G_{b}$-metric spaces, Adv. Differ. Equ. 2020, 152 (2020).
https://doi.org/10.1186/s13662-020-02605-0

\bibitem{O Ege 2019} M. Ramezani, \"{O}. Ege and M. De la Sen, A new fixed
point theorem and a new generalized Hyers-Ulam-Rassias stability in
incomplete normed spaces, Mathematics 2019, 7, 1117.
https://doi.org/10.3390/math7111117

\bibitem{ege} M. E. Ege and C. Alaca, \textit{Fixed point results and an
application to homotopy in modular metric spaces}, J. Nonlinear Sci. Appl. 8
(6) (2015) 900-908.

\bibitem{gupta} A. Gupta, \textit{Cyclic Contraction on }$S$-\textit{Metric
Space,} International Journal of Analysis and Applications 3 (2) (2013)
119-130.

\bibitem{khan} K. A. Khan, \textit{Generalized Normed Spaces and Fixed Point
Theorems, }Journal of Mathematics and Computer Science 13 (2014) 157-167.

\bibitem{liu} Z. Liu, Y. Xu and Y. J. Cho, \textit{On Characterizations of
Fixed and Common Fixed Points, }J. Math. Anal. Appl. 222 (1998) 494-504.

\bibitem{Mohanta} S. K. Mohanta, \textit{Some fixed point theorems in }$G$%
\textit{-metric spaces}, An. \c{S}tiin\c{t}. Univ.
"Ovidius\textquotedblright\ Constan\c{t}a Ser. Mat. 20 (1) (2012) 285-305.

\bibitem{Mustafa} Z. Mustafa and B. Sims, \textit{A new approach to
generalized metric spaces}, J. Nonlinear Convex Anal. 7 (2) (2006) 289-297.

\bibitem{paula} P. Oliveria, \textit{Two Results on Fixed Points, }%
Proceedings of the Third World Congress of Nonlinear Analysts, Part 4
(Catania, 2000). Nonlinear Anal. 47 (2001) 2703-2717.

\bibitem{nihal} N. Y. \"{O}zg\"{u}r and N. Ta\c{s}, \textit{Some Fixed Point
Theorems on }$S$-\textit{Metric Spaces, }Mat. Vesnik 69 (1) (2017) 39-52.

\bibitem{nihal2} N. Y. \"{O}zg\"{u}r and N. Ta\c{s}, \textit{Some New
Contractive Mappings on }$S$-Metric Spaces and Their Relationships with the
Mapping $(\mathbf{S25})$, Math. Sci. 11 (1) (2017) 7-16.

\bibitem{rhoades} B. E. Rhoades, \textit{A Comparison of Various Definitions
of Contractive Mappings, }Trans. Amer. Math. Soc. 226 (1977) 257-290.

\bibitem{sedghi2} S. Sedghi, N. Shobe and A. Aliouche, \textit{A
Generalization of Fixed Point Theorems in }$S$-\textit{Metric Spaces}, Mat.
Vesnik 64 (3) (2012) 258-266.

\bibitem{smulian} V. Smulian, \textit{On The Principle of Inclusion in The
Space of The Type }$($\textit{B}$)$, Rec. Math. [Math. Sbornik] 5 (47) (2)
(1939) 317-328.

\bibitem{PhD} N. Ta\c{s}, Fixed point theorems and their various
applications, Ph. D. Thesis, 2017.
\end{thebibliography}
\end{document}